\documentclass{birkjour}

\usepackage{amsmath,amssymb,amsthm}
\usepackage{color}

\newcommand{\R}{\mathbb{R}}

\newcommand{\N}{\mathbb{N}}
\newcommand{\Z}{\mathbb{Z}}
\renewcommand{\S}{\mathcal{S}}
\newcommand{\F}{\mathcal{F}}

\newcommand{\supp}{\mathop{\mathrm{supp}}}
\newcommand{\grad}{\mathop{\mathrm{grad}}}
\newcommand{\sgn}{\mathop{\mathrm{sgn}}}
\renewcommand{\i}{\mathbf{i}}

 \newtheorem{thm}{Theorem}[section]
 \newtheorem{cor}[thm]{Corollary}
 \newtheorem{lem}[thm]{Lemma}
 \newtheorem{prop}[thm]{Proposition}
 \theoremstyle{definition}
 \newtheorem{defn}[thm]{Definition}
 \theoremstyle{remark}
 \newtheorem{rem}[thm]{Remark}
 
 \numberwithin{equation}{section}

\begin{document}

%-------------------------------------------------------------------------
% editorial commands: to be inserted by the editorial office
%
%\firstpage{1} \volume{228} \Copyrightyear{2004} \DOI{003-0001}
%\seriesextraline{This is the Concrete Title of this Book\br H.E. R and S.T.C. W, Eds.}
%\firstpage{1}

\title[Superposition Operators on Weighted Modulation Spaces]
 {Non-analytic Superposition Results on Modulation Spaces with Subexponential Weights}

%----------Author 1
\author{Maximilian Reich}

\address{%
Fakult\"at f\"ur Mathematik und Informatik\\
TU Bergakademie Freiberg, Germany}

\email{maximilian.reich@math.tu-freiberg.de}

%\thanks{This work was completed with the support of our
%----------Author 2
\author{Michael Reissig}
\address{%
Fakult\"at f\"ur Mathematik und Informatik\\
TU Bergakademie Freiberg, Germany}

\email{reissig@math.tu-freiberg.de}
%----------Author 3
\author{Winfried Sickel}
\address{%
Mathematisches Institut \\
Friedrich-Schiller-Universit\"at Jena, Germany}
\email{winfried.sickel@uni-jena.de}
\subjclass{}

\keywords{modulation spaces, subexponential weights, Gevrey spaces, multiplication algebras, superposition operators}

\date{April 1, 2015}

\begin{abstract}
Motivated by classical results for Gevrey spaces and their applications to nonlinear partial differential 
equations we define so-called Gevrey-modulation spaces. 
We establish analytic as well as non-analytic superposition results on Gevrey-modulation spaces. 
These results are extended to a special weighted  modulation space where the weight increases stronger than any polynomial but less than 
as in the  Gevrey case.
\end{abstract}

%%% ----------------------------------------------------------------------
\maketitle
%%% ----------------------------------------------------------------------
%\tableofcontents

%&&&&&&&&&&&&&&&&&&&&&&&&&&&&&&&&&&&&&&&&&&&&&&
%&&&&&&&&&&&&&&&&&&&&&&&&&&&&&&&&&&&&&&&&&&&&&&

\section{Introduction} \label{introduction}

%&&&&&&&&&&&&&&&&&&&&&&&&&&&&&&&&&&&&&&&&&&&&&&
%&&&&&&&&&&&&&&&&&&&&&&&&&&&&&&&&&&&&&&&&&&&&&&

Gevrey analysis is an effective tool to treat several models of partial differential equations. 
Instead of treating the model in the physical space considerations in the phase space are more appropriate. 
For us, a Gevrey function can be characterized by it's behavior on the Fourier transform side, i.e.,
	\[ f\in \mathcal{G}_s \qquad \Longleftrightarrow \qquad e^{\langle \xi \rangle^{\frac{1}{s}}} \F f(\xi) \in L^2(\R^n), \]
	where $s>1$.\\	
In \cite{brs} the Gevrey example is considered, that is, the Cauchy problem
	\begin{equation} \label{Gevrey1}
		u_{tt} - u_x = 0, \quad u(0,x) = \phi(x), \quad u_t(0,x) = \psi(x).
	\end{equation}
It is globally (in time) well-posed in the Gevrey space $\mathcal{G}_s$ if and only if $s < 2$. \\
Another Cauchy problem which is reasonable to consider in Gevrey spaces is the following one:
	\begin{equation} \label{Gevrey3}
		u_{tt} - a(t)u_{xx} = 0, \quad u(0,x) = \phi(x), \quad u_t(0,x) = \psi(x).
	\end{equation}
	If we suppose that the positive coefficient $a=a(t)$ belongs to the 
H\"older space  $C^{\alpha}[0,T]$, $0<\alpha<1$, then \eqref{Gevrey3} is globally (in time) well-posed in the 
Gevrey space $\mathcal{G}_s$ for $s < \frac{1}{1-\alpha}$ which is explained in \cite{colombini}.\\
If one is interested in the solvability behavior of the corresponding semi-linear Cauchy problems
\begin{equation} \label{Gevrey2}
		u_{tt} - u_x = f(u), \quad u(0,x) = \phi(x), \quad u_t(0,x) = \psi(x),
	\end{equation}
or	
\begin{equation} \label{Gevrey4}
		u_{tt} - a(t)u_{xx} = f(u), \quad u(0,x) = \phi(x), \quad u_t(0,x) = \psi(x)
	\end{equation}
with an admissible nonlinearity $f=f(u)$ and Gevrey data $\phi$ and $\psi$, then one of the first steps is to explain superposition 
operators in $\mathcal{G}_s$. This was done in 	
\cite{brs}. There appropriate superposition operators are studied in spaces which are defined by the behavior of the Fourier transform 
as in case of  $\mathcal{G}_s$. \\
In the present paper we devote ourselves to modulation spaces. In various papers Wang et all \cite{wangExp,wang,wang1,wang2} have shown 
that modulation spaces may serve as a reasonable tool when studying existence and regularity of linear and nonlinear partial 
differential equations.
It will turn out that modulation spaces equipped with some admissible subexponential weights allow to prove for superposition operators 
similar results as in \cite{brs}. \\
The paper is organized as follows. First of all an approach to modulation spaces $M_{p,q}^s(\R^n)$ and Gevrey-modulation spaces 
$\mathcal{GM}_{p,q}^s(\R^n)$ is chosen as introduced in \cite{wangExp}. After obtaining boundedness of functions in 
Gevrey-modulation spaces we are interested in investigating the behavior of analytic nonlinearities. Therefore it is 
sufficient to prove algebra properties which is done  in Section \ref{Modulationspaces}.\\
In Section \ref{Gevreysuperposition} we are able to find non-analytic functions $f\in C^\infty(\R^n)$ such that the 
corresponding superposition operator $T_f$ defined by
\[ T_f:  u \to T_f u:= f(u)\] maps the Gevrey-modulation space $\mathcal{GM}_{p,q}^s(\R^n)$ into itself.\\
In Section \ref{Ultradifferentiablesuperposition} we shall study superposition operators in a special modulation 
space of ultra-differentiable functions with another type of subexponential weight. This space contains all Gevrey spaces.\\
The main concern of this paper is to prove some analytic as well as non-analytic superposition results in 
particularly weighted modulation spaces.\\
Some open problems and concluding remarks complete the paper (see Section \ref{finalcomments}).

%&&&&&&&&&&&&&&&&&&&&&&&&&&&&&&&&&&&&&&&&&&&&&&&&&&&&&
%&&&&&&&&&&&&&&&&&&&&&&&&&&&&&&&&&&&&&&&&&&&&&&&&&&&&&

\section{Modulation Spaces} \label{Modulationspaces}

%&&&&&&&&&&&&&&&&&&&&&&&&&&&&&&&&&&&&&&&&&&&&&&&&&&&&&
%&&&&&&&&&&&&&&&&&&&&&&&&&&&&&&&&&&&&&&&&&&&&&&&&&&&&&

%&&&&&&&&&&&&&&&&&&&&&&&&&&&&&&&&&&&&&&&&&&&&&&
%&&&&&&&&&&&&&&&&&&&&&&&&&&&&&&&&&&&&&&&&&&&&&&

\subsection{Definitions} \label{Definitions}

%&&&&&&&&&&&&&&&&&&&&&&&&&&&&&&&&&&&&&&&&&&&&&&&
%&&&&&&&&&&&&&&&&&&&&&&&&&&&&&&&&&&&&&&&&&&&&&&&

First of all we introduce some basic notations and definitions. In $\R^n$ the notation of 
multi-indices $\alpha = (\alpha_1, \ldots , \alpha_n)$ is used, where $|\alpha|=\sum_{j=1}^n \alpha_j$.
 Given two multi-indices $\alpha$ and $\beta$, then $\alpha \leq \beta$ means $\alpha_j \leq \beta_j$ for $1\leq j\leq n$. 
Furthermore let $f$ be a function on $\R^n$ and $x\in\R^n$, then
	\[ x^\alpha = \prod_{j=1}^n x_j^{\alpha_j} \]
	and
\[ D^\alpha f(x) = \frac{1}{\i^{|\alpha|}} \frac{\partial^\alpha}{\partial x^\alpha} f(x) = 
\frac{1}{\i^{|\alpha|}} \Big(\prod_{j=1}^n \frac{\partial^{\alpha_j}}{\partial x_j^{\alpha_j}} \Big)f(x). 
\]
	A function $f\in C^\infty (\R^n)$ belongs to the Schwartz space $\S(\R^n)$ if and only if
		\[ \sup_{x\in \R^n} | x^\alpha D^\beta f(x)| < \infty \]
		for all multi-indices $\alpha, \beta$. The set of all tempered distributions is denoted by $\S'(\R^n)$ which is the dual space of 
		$\S(\R^n)$. \\
We introduce $\langle \xi\rangle_m^s := (m^2+|\xi|^2)^{\frac{s}{2}}$.
		If $m=1$, then we write $\langle \xi \rangle^s$ for simplicity. \\
		The notation $a \lesssim b$ is equivalent to $a \leq Cb$ with a positive constant $C$. Moreover, by writing $a \asymp b$ we mean $a \lesssim b \lesssim a$. The Fourier transform of 
an admissible function $f$ is defined by
		\[ \F f(\xi) = \hat{f} (\xi) = (2\pi)^{-\frac{n}{2}} \int_{\R^n} f(x) e^{-\i x\cdot \xi} \, dx \quad (x,\xi \in\R^n). \]
		Analogously the inverse Fourier transform is defined by
		\[ \F^{-1} \hat{f}(x) = (2\pi)^{-\frac{n}{2}} \int_{\R^n} \hat{f}(\xi) e^{\i x\cdot \xi} \, d\xi \quad (x,\xi \in\R^n). \]		 
		In order to describe local frequency properties of a function $f$ we define the following joint time-frequency representation.
		\begin{defn} \label{STFT}
			Let $\phi$ be the so-called window function which is a fixed function that is not identically zero. 
Then the short-time Fourier transform (STFT) of a function $f$ with respect to $\phi$ is defined as
			\[ V_\phi f(x,\xi) = (2\pi)^{-\frac{n}{2}} \int_{\R^n} f(s) \overline{\phi(s-x)} e^{-\i s\cdot \xi} ds 
\qquad (x,\xi \in \R^n). \]
		\end{defn}
%It is needed to choose sufficiently smooth window functions to avoid artificial discontinuities 
%of the corresponding STFT $V_\phi f$. The regularity of window functions will be mentioned particularly in subsequent 
%definitions. 

We want to define a family of Banach spaces which controls globally the joint time-frequency 
information. Therefore we introduce weighted modulation spaces, where we will use particular weights. A more detailed 
discussion about weight functions can be found in Chapter 11 in \cite{groechenig}.

\begin{defn} \label{modCont}
Let $1\leq p,q\leq\infty$. Let $\phi\in \S(\R^n)$ be a 
fixed window and assume $s,\sigma \in\R$ to be the weight parameters. Then the weighted modulation 
space $\mathring{M}_{p,q}^{s,\sigma}(\R^n)$ is the set
			\[ \mathring{M}_{p,q}^{s,\sigma} (\R^n) := \{ f\in \S'(\R^n): \|f\|_{\mathring{M}_{p,q}^{s,\sigma}} < \infty \}, \]
			where the norm is defined as
			\[ \|f\|_{\mathring{M}_{p,q}^{s,\sigma}} = \Big( \int_{\R^n} 
\Big( \int_{\R^n} |V_\phi f(x,\xi) \langle x\rangle^\sigma \langle \xi\rangle^s|^p dx \Big)^{\frac{q}{p}} d\xi \Big)^{\frac{1}{q}}. \]
			For $p=\infty$ and/or $q=\infty$ the definition can be obviously modified by taking $L^\infty$-norms.
		\end{defn}
\begin{rem}
(i) 
If $s=\sigma=0$, then we obtain the so-called standard modulation space $\mathring{M}_{p,q}(\R^n)$. 
For $\sigma=0$, i.e., no weight with respect to the $x$-variable, the weighted modulation space is denoted by $\mathring{M}_{p,q}^s(\R^n)$. 
Subsequently the space $\mathring{M}_{p,q}^{s,\sigma}(\R^n)$ is just referred to as modulation space unless it is explicitly stated differently. 
\\
%Furthermore, if $p=q$ we only write $\mathring{M}_p^{s,\sigma}(\R^n)$. 
%Moreover, remark that the behavior of functions $f\in \mathring{M}_{p,q}^{s,\sigma}(\R^n)$ can be controlled by the weights. 
(ii) A rough interpretation is as follows.
The weight in $x$ in the preceding definition corresponds to some growth or decay properties of $f$. On the other hand 
the weight in $\xi$ corresponds to regularity properties of $f$ in $\mathring{M}_{p,q}^{s,\sigma} (\R^n)$.
\\
(iii) General references with respect to (weighted) modulation spaces are 
Feichtinger \cite{feichtinger}, Groechenig \cite{groechenig}, Toft \cite{toftConv}, Triebel \cite{trzaa} and
Wang et. all \cite{wangExp} to mention only a few. 
\end{rem}

At this point we want to go back to the alternative approach to the STFT. 
Since we aim at specific superposition results on weighted modulation spaces it will turn out that introducing 
the following approach to the STFT will be convenient. We are basically adopting the idea of obtaining local frequency properties 
of a function $f$. 
Related  frequency decomposition techniques are explained in \cite{groechenig}. A special case, the so-called frequency-uniform 
decomposition, was independently introduced by Wang (e.g., see \cite{wang}). Let $\rho: \R^n \mapsto [0,1]$ be a Schwartz function which 
is compactly supported in the cube 
\[
Q_0 := \{ \xi \in \R^n: -1\leq \xi_i \leq 1,\:  i=1,\ldots,n \}\, .
\] 
Moreover, 
\[
\rho(\xi)=1 \qquad \mbox{if}  \quad  |\xi_i|\leq \frac{1}{2}\, , \qquad i=1, 2, \ldots \, n. 
\]
With $\rho_k (\xi) :=\rho (\xi-k) $, $\xi \in \R^n$, $k \in \Z^n$, it follows
\[
\sum_{k \in \Z^n} \rho_k (\xi)\ge 1 \qquad \mbox{for all}\quad \xi \in \R^n\, .
\]
Finally we define
\[ 
\sigma_k(\xi) := \rho_k(\xi) \Big(\sum_{k\in\Z^n} \rho_k(\xi)\Big)^{-1}, \qquad \xi \in \R^n\, , \quad k\in\Z^n \, .
\]
The following  properties are obvious: 
		\begin{itemize}
			\item $ 0 \le \sigma_k(\xi) \le 1$ for all $\xi \in \R^n$;
			\item $\supp \sigma_k \subset Q_k := \{ \xi \in \R^n: -1\leq \xi_i -k_i \leq 1, \: i=1,\ldots,n \} $;
			\item $\displaystyle \sum_{k\in\Z^n} \sigma_k(\xi) \equiv 1$ for all $\xi\in\R^n$;
			\item There exists a constant $C>0$ such that $\sigma_k (\xi) \ge C$ if $ \max_{i=1, \ldots \, n}\, |\xi_i-k_i|\leq \frac{1}{2}$; 
			\item For all $m \in \N_0$ there exist positive constants $C_m$ such that for $|\alpha|\leq m$
\[
\sup_{k \in \Z^n}\, \sup_{\xi \in \R^n} \, |D^\alpha \sigma_k(\xi)|\leq C_m\,   .
\]
		\end{itemize}
		The operator
		\[ \Box_k := \F^{-1} \left( \sigma_k \F (\cdot) \right), \quad k\in\Z^n, \]
		is called uniform decomposition operator. \\

As it is well-known there is an equivalent description of the 
modulation spaces by means of the uniform decomposition operator, see Feichtinger \cite{feichtinger}.

\begin{defn} \label{defdecomp}
Let $1\leq p,q \leq \infty$ and assume $s \in\R$ to be the weight parameter. Suppose the window $\rho\in \S(\R^n)$ is compactly supported. 
Then the weighted modulation space $M_{p,q}^{s}(\R^n)$ consists of all tempered distributions $f\in \S'(\R^n)$ such that their norm
\[ \|f\|_{M_{p,q}^{s}} = \Big( \sum_{k\in\Z^n} \langle k \rangle^{sq} \|\Box_k f\|_{L^p}^q \Big)^{\frac{1}{q}} \]
is finite with obvious modifications when $p=\infty$ and/or $q=\infty$.
\end{defn}

\begin{prop} \label{normequivalence}
The norms of the Definitions \ref{modCont} and \ref{defdecomp} are equivalent. Let $f\in \S'(\R^n)$. Then it holds
		\[ C_1 \|f \|_{\mathring{M}_{p,q}^{s}} \leq \|f\|_{M_{p,q}^{s}} \leq C_2 \|f\|_{\mathring{M}_{p,q}^{s}}, \]
		where the positive constants $C_1$ and $C_2$ are depending on the dimension $n$, the window function and on the 
frequency-uniform decomposition, respectively.
	\end{prop}
	\begin{proof}
		Cf. Proposition 2.1 in \cite{wang} and Proposition 1.12 in \cite{reich}.
	\end{proof}

In what follows we shall always work with the uniform decomposition operator.

%&&&&&&&&&&&&&&&&&&&&&&&&&&&&&&&&&&&&&&&&&&&&&&&&&&&&&&&&&&&&&&&&&&&&&&&&&&&
%&&&&&&&&&&&&&&&&&&&&&&&&&&&&&&&&&&&&&&&&&&&&&&&&&&&&&&&&&&&&&&&&&&&&&&&&&&&

	\subsection{Gevrey-Modulation Spaces} \label{Gevreymodulationspace}

%&&&&&&&&&&&&&&&&&&&&&&&&&&&&&&&&&&&&&&&&&&&&&&&&&&&&&&&&&&&&&&&&&&&&&&&&&&&
%&&&&&&&&&&&&&&&&&&&&&&&&&&&&&&&&&&&&&&&&&&&&&&&&&&&&&&&&&&&&&&&&&&&&&&&&&&&

%So far we got a rough sketch of the concept of  weighted modulation spaces $M_{p,q}^s(\R^n)$. Our goal 
%in this paper will consist in  showing non-analytic superposition results. 
There is a famous classical result by Katznelson \cite{Ka} (in the periodic case) and by 
Helson, Kahane,  Katznelson,  Rudin \cite{HKKR} (nonperiodic case) 
which says that only analytic functions operate on the Wiener algebra $\mathcal{A}(\R^n)$.
More exactly, the operator $T_f : ~ u \mapsto f(u)$ maps $\mathcal{A}(\R^n)$ into $\mathcal{A}(\R^n)$ if and only 
if $f(0)=0$ and $f$ is analytic. Here $\mathcal{A}(\R^n)$ is the collection of all $u \in C (\R^n)$ such that 
$\F u \in L^1 (\R^n)$. Moreover, a similar result is obtained for particular standard modulation spaces. 
In \cite{Bihami2} is stated that $T_f$ maps $M_{1,1}$ into $M_{1,1}$ if and only if $f(0)=0$ and $f$ is analytic. 
Therefore, the existence of non-analytic superposition results for weighted modulation spaces  is a priori not so clear.
Here we are interested in weigthed modulation spaces with different weights than used above.

\begin{defn} \label{modExpWeight}
The integrability parameters are given by $1\leq p,q \leq \infty$. Let $\rho\in C_0^\infty (\R^n)$ be a fixed window and 
assume $s > 0$ to be the weight parameter. By $(\sigma_k)_k$ we denote the associated uniform decomposition of unity
in the above sense.
Then the Gevrey-modulation space $\mathcal{GM}_{p,q}^{s}(\R^n)$ is the collection of all $f \in L^p (\R^n)$
such that 
\[ \|f\|_{\mathcal{GM}_{p,q}^{s}} = \Big( \sum_{k\in\Z^n} e^{q |k|^{\frac{1}{s}}} \, \|\Box_k f\|_{L^p}^q \Big)^{\frac{1}{q}} < \infty \]
(with obvious modifications when $p=\infty$ and/or $q=\infty$).
\end{defn}

\begin{rem}
 \rm
(i) Modulation spaces with general weights have been also considered by Gol'dman \cite{go}
(general function spaces of Besov type) and Triebel \cite{trzaa} (trace problems).
\\
(ii) We shall call the weights $w(x):= e^{|x|^{\frac{1}{s}}} $, $x \in \R^n$,  with $s>1$ subexponential.
\end{rem}

It is not difficult to prove the following basic facts. 

\begin{lem}\label{basic}
(i) The Gevrey-modulation space $\mathcal{GM}_{p,q}^{s}(\R^n)$ is a Banach space.
\\
(ii) $\mathcal{GM}_{p,q}^{s}(\R^n)$ is independent of the choice of the window $\rho\in C_0^\infty (\R^n)$ in the sense of equivalent norms.
\\
(iii) $\mathcal{GM}_{p,q}^{s}(\R^n)$ has the Fatou property, i.e., if $(f_m)_{m=1}^\infty \subset \mathcal{GM}_{p,q}^{s}(\R^n)$ 
is a sequence such that 
$ f_m \rightharpoonup f $ (weak convergence in $S'(\R^n)$) and 
\[
 \sup_{m \in \N}\,  \|\, f_m \, \|_{\mathcal{GM}_{p,q}^{s}} < \infty\, , 
\]
then $f \in \mathcal{GM}_{p,q}^{s}$ follows and 
\[
\|\, f \, \|_{\mathcal{GM}_{p,q}^{s}} \le \sup_{m\in \N}\,  \|\, f_m \, \|_{\mathcal{GM}_{p,q}^{s}} < \infty\, , 
\]
\end{lem}

\begin{proof}
It is enough to comment on a proof of (iii).
We follow \cite{fr}. 
>From assumption, it follows that for all $k\in \Z^n$ and $x\in\R^n$,
\begin{eqnarray*}
\F^{-1} \, [\sigma_k \, \F f_m](x)=  (2 \pi)^{-n/2} \, f_m(x-\cdot)(\sigma_k) \to f(x-\cdot)(\sigma_k)= 
\F^{-1} \, [\sigma_k \, \F f](x) 
\end{eqnarray*}
as $m\to\infty$. Fatou's lemma yields
\begin{eqnarray*}
\sum_{|k| \le N}  \Big(\int_{\R^n} \, |\F^{-1} \, [\sigma_k \, \F f](x)|^p dx\, \Big)^{\frac qp} 
\le \liminf_{m\to\infty} \sum_{|k|\le N} \, \Big(\int_{\R^n} \,  |\F^{-1} \, [\sigma_k \, \F f_m](x)|^p dx \, \Big)^{\frac qp} \, .
\end{eqnarray*}
An obvious monotonicity argument completes the proof.
\end{proof}

Obviously the  spaces $\mathcal{GM}_{p,q}^{s}(\R^n)$ are monotone in $s$ and $q$. But they are also monotone with respect to $p$. 
To show this we recall Nikol'skijs inequality, see, e.g.,  Nikol'skij \cite[3.4]{Ni} or Triebel \cite[1.3.2]{triebel}.

\begin{lem} \label{nikolskij}
Let $1\leq p \leq q \leq\infty$ and $f$ be an integrable function with $\supp \F f(\xi) \subset B(y,r)$, i.e., 
the support of the Fourier transform of $f$ is contained in a ball with radius $r>0$ and center in $y \in \R^n$. Then it holds
		\[ \|f\|_{L^q} \leq C r^{n(\frac{1}{p}-\frac{1}{q})} \|f\|_{L^p} \]
with a constant $C>0$ independent of $r$ and $y$.
\end{lem}

This implies $\|\Box_k f\|_{L^q} \le c \, \|\Box_k f\|_{L^p}$ if $p \le q$ with $c$ independent of $k$ and $f$ which results in the following
corollary.

\begin{cor}\label{einbettung}
Let $0 < s_0 < s$, $p_0 < p $ and $q_0 < q$.
Then the following  embeddings hold and are continuous:
\[
\mathcal{GM}_{p,q}^{s_0}(\R^n) \hookrightarrow \mathcal{GM}_{p,q}^{s}(\R^n)\, , \qquad \mathcal{GM}_{p_0,q}^{s}(\R^n) \hookrightarrow \mathcal{GM}_{p,q}^{s}(\R^n) 
\]
and
\[
 \mathcal{GM}_{p,q_0}^{s}(\R^n) \hookrightarrow \mathcal{GM}_{p,q}^{s}(\R^n)\, ; 
\]
i.e., for all $p,q$, $1\le p,q\le \infty$ we have 
\[
\mathcal{GM}_{1,1}^{s}(\R^n) \hookrightarrow 
\mathcal{GM}_{p,q}^{s}(\R^n) \hookrightarrow  \mathcal{GM}_{\infty,\infty}^{s}(\R^n)\, .
\]
\end{cor}

Only very smooth functions have a chance to belong to one of the spaces $\mathcal{GM}_{p,q}^{s}(\R^n)$. 

\begin{cor}\label{gevrey}
Let $s>0$ and $1\le p,q\le \infty$.
If $f \in \mathcal{GM}_{p,q}^{s}(\R^n)$, then $f$ is infinitely often differentiable and there exists a constant $C=C(f,n)$ such that
\begin{equation}\label{ws-02}
 |D^\alpha f (x)| \le C^{|\alpha|+1} \, (\alpha !)^{sn}\, , \qquad x \in \R^n\, .
\end{equation}
\end{cor}

\begin{proof}
Let $\varphi \in C_0^\infty (\R^n)$ be a function such that $\varphi :~ \R^n \to [0,1]$, 
$\varphi (\xi)= 1 $ if $\xi \in \supp \rho$ and  $\supp \varphi \subset \max_{i=1, \ldots \, n} |\xi_i|\le 2$.
We put $\varphi_k(\xi) := \varphi (\xi-k)$, $\xi \in \R^n$, $k \in \Z^n$.
In what follows we work with the distributional derivative, i.e., the assumption $f \in S'(\R^n)$ is sufficient.
Hence, if $m >n/2$, there exists a constant $c_1$ such that
\[
 |\Box_k (D^\alpha f) (x)| = |\F^{-1}[ \varphi_k (\xi) \, \xi^\alpha \, \sigma_k (\xi)\,  \F f (\xi)](x)|\le 
 c_1\, \| \, \varphi_k (\xi) \, \xi^\alpha \|_{W^m_2} \, |\Box_k f (x)|
\]
holds for all $f$ and all $x$, see, e.g., \cite[Prop.~1.7.5]{ST}.
By using
\begin{eqnarray*}
 \| \, \varphi_k (\xi) \, \xi^\alpha \|_{W^m_2} & \le & |\supp \varphi |^{1/2} \, \max_{|\gamma|\le m} \sup_{\xi \in \supp \varphi_k}\,  |D^\gamma (\varphi_k (\xi)\, \xi^\alpha) \, |
\\
& \le & c_2 \, (1+ |k|)^{|\alpha|}
\end{eqnarray*}
with $c_2$ independent of $k$ and $\alpha$ we conclude
\[
 |\Box_k (D^\alpha f) (x)| \le c_3 \, (1+ |k|)^{|\alpha|} \, |\Box_k  f (x)|\, , \qquad x \in \R^n\, .
\]
Hence, for any $N$ we have
\begin{eqnarray}\label{ws-01}
\sum_{|k|<N} |\Box_k (D^\alpha f) (x)| & \le &  c_3 \, \sum_{|k|<N}  (1+ |k|)^{|\alpha|} \, \|\, \Box_k  f (x)\, \|_{L^\infty}
\nonumber
\\
& \le & c_3 \, \Big(\sum_{|k|<N}  (1+ |k|)^{|\alpha|} \,  e^{-|k|^{1/s}} \Big)  \, \|\, f \, \|_{\mathcal{GM}_{\infty,\infty}^{s}}
\nonumber
\\
& \le & c_4  \, \|\, f \, \|_{\mathcal{GM}_{\infty,\infty}^{s}}
\, , \qquad x \in \R^n\, ,
\end{eqnarray}
where $c_4$ does not depend on $f,N$ and $x$.
This implies convergence of \\
$\displaystyle \Big(\sum_{|k|<N} \Box_k (D^\alpha f) \Big)_N$ in $C_{ub}(\R^n)$. Since
\[
\lim_{N \to \infty} \, \sum_{|k|<N} \Box_k (D^\alpha f) = D^\alpha f  \qquad \mbox{in}\quad S'(\R^n)
\]
we conclude $D^\alpha f \in C_{ub}(\R^n)$ and
\[
|D^\alpha f (x)| \le \sum_{k\in \Z^n} |\Box_k (D^\alpha f) (x)| \le c_4\, \| \, f \, \|_{\mathcal{GM}_{\infty,\infty}^{s}}
\, , \qquad x \in \R^n\, .
\]
This proves $f \in C^\infty (\R^n)$. Now we turn back to the estimate \eqref{ws-01}.
Observe
\begin{eqnarray*}
\sum_{|k|<N}  (1+ |k|)^{|\alpha|} \,  e^{-|k|^{1/s}} &\asymp & \int_{\R^n}  (1+ |x|)^{|\alpha|} \,  e^{-|x|^{1/s}}\, dx
\\
& =& c_n \, \int_{0}^\infty   (1+r)^{|\alpha|+n-1} \,  e^{-r^{1/s}}\, dr
\\
&  \asymp &  \int_{0}^\infty   t^{s(|\alpha|+n)-1} \,  e^{-t}\,   dt = \Gamma (s(|\alpha|+n)),
\end{eqnarray*}
where the constants behind $\asymp$ are independent of $\alpha$.
Next we apply\\
 $\Gamma (x)\le x^{x-1}$. This yields
\[
|D^\alpha f (x)|\le c_5 \, (s(|\alpha|+n))^{ (s(|\alpha|+n))-1}\, \| \, f \, \|_{\mathcal{GM}_{\infty,\infty}^{s}}\, .
\]
Recall the multinomial theorem
\[
(\alpha_1 + \ldots + \alpha_n)^k = \sum_{|\beta|=k} \, {k \choose \beta} \, \alpha^\beta
\]
and take into account that
\[
\max_{|\beta| = |\alpha|} \, \alpha^\beta = \alpha ^\alpha 
\]
we obtain with $k = |\alpha|+1$
\begin{eqnarray*}
  (s(|\alpha|+n))^{ (s(|\alpha|+n))} & \le  &   (sn)^{sn}\, (sn)^{|\alpha|}\,  \Big((|\alpha|+1)^{|\alpha|+1}\Big)^{sn} 
\\
& \le & (sn)^{sn}\, (sn)^{|\alpha|}\,  \Big(\gamma^\gamma \,  \sum_{|\beta|=k} \, {k \choose \beta}\Big)^{sn}
\\
& \le & (sn)^{sn}\, (sn)^{|\alpha|}\, n^{(|\alpha|+1)sn} \,  (\gamma^\gamma) ^{sn} \, .  
\end{eqnarray*}
Here $\gamma = \alpha + e^j$, where $e^j= (0, \ldots , 0,1,0, \ldots \, , n)$ for some $j \in \{1,2, \ldots \, n\}$.
In case $\alpha_1 \neq 0$ we find
\begin{eqnarray*}
 (s(|\alpha|+n))^{ (s(|\alpha|+n))} & \le &  (sn)^{sn}\, (sn)^{|\alpha|}\, n^{(|\alpha|+1) sn} \, 
\Big((\alpha_1 + 1)^{\alpha_1 + 1} \, \prod_{j=2}^n \alpha_j^{\alpha_j}\Big)^{sn} 
\\
&\le &  (sn)^{sn}\, (sn)^{|\alpha|}\, (2n)^{(|\alpha|+1) sn} \, \Big(\alpha_1  \, \prod_{j=1}^n \alpha_j^{\alpha_j}\Big)^{sn} 
\\
&\le &  (sn)^{sn}\, (sn)^{|\alpha|}\, (2n)^{(|\alpha|+1) sn} \, |\alpha|^{sn}\,  (\alpha^\alpha)^{sn} 
\, .  
\end{eqnarray*}
In case $\alpha_1 =0$ we have to modify this argument in an obvious way.
Inserting the obtained estimate in our previously found inequality  we have proved
\[
|D^\alpha f (x)|\le c_5 \, (sn)^{sn}\, (sn)^{|\alpha|}\, (2n)^{(|\alpha|+1) sn} \, |\alpha|^{sn}\,  (\alpha^\alpha)^{sn}  \, \| \, f \, \|_{\mathcal{GM}_{\infty,\infty}^{s}}\, .
\]
An application of Stirling's formula yields the claim.
\end{proof}

\begin{rem}
 \rm
(i) Inequality \eqref{ws-02} makes clear that the elements of the spaces $\mathcal{GM}_{p,q}^{s}(\R^n)$ have some 
classical Gevrey regularity, see, e.g., Rodino \cite[Definition 1.4.1]{Ro}.
In fact, we proved $\mathcal{GM}_{p,q}^{s}(\R^n) \subset G^{sn} (\R^n)$.
\\
(ii)
The Corollary  \ref{gevrey} remains true  under the weaker assumption $f \in S'(\R^n)$ and $\| \, f \, \|_{\mathcal{GM}_{\infty,\infty}^{s}} < \infty$.
As a consequence we observe that a replacement of the requirement $f \in L^p(\R^n)$
by $f \in S'(\R^n)$ in Definition  \ref{modExpWeight} does not change the space $\mathcal{GM}_{p,q}^{s}(\R^n)$.
\end{rem}

For $p=q=2$ we can simplify the description of $\mathcal{GM}_{p,q}^{s}(\R^n)$.

\begin{lem}\label{gleich}
 \rm
 A tempered distribution $f$ belongs to $\mathcal{GM}_{2,2}^{s}(\R^n)$ if and only if 
 $e^{|\, \cdot \, |^{1/s}} \, \F f (\, \cdot \, ) \in L^2(\R^n)$.
\end{lem}

\begin{proof}
 This follows from
 \begin{eqnarray*}
\| \, f \, \|_{\mathcal{GM}_{2,2}^{s}}^2 & = &    \sum_{k\in\Z^n} e^{2 |k|^{\frac{1}{s}}} \|\Box_k f\|_{L^2}^2
\\
&= &  \sum_{k\in\Z^n} e^{ 2|k|^{\frac{1}{s}}} \int |\sigma_k (\xi)\,  \F f (\xi) |^2 d\xi 
\\
& \asymp  &  \sum_{k\in\Z^n} e^{2 |k|^{\frac{1}{s}}} \int_{\|\xi - k\|_\infty\le 1} |\F f (\xi) |^2 d\xi 
\\
& \asymp  &  \sum_{k\in\Z^n}  \int_{\|\xi - k\|_\infty\le 1} e^{2 |\xi|^{\frac{1}{s}}}\,  |  \F f (\xi) |^2 d\xi  \asymp   
 \int_{\R^n} e^{2 |\xi|^{\frac{1}{s}}}\,  |  \F f (\xi) |^2 d\xi\, , 
 \end{eqnarray*}
 where we used the properties of our decomposition of unity.

\end{proof}

%&&&&&&&&&&&&&&&&&&&&&&&&&&&&&&&&&&&&&&&&&&&&&&&&&&&&&&&&&&&&&&&&&&&&&&&&&&&
%&&&&&&&&&&&&&&&&&&&&&&&&&&&&&&&&&&&&&&&&&&&&&&&&&&&&&&&&&&&&&&&&&&&&&&&&&&&

\subsection{Multiplication Algebras} \label{multiplicationalgebras}

%&&&&&&&&&&&&&&&&&&&&&&&&&&&&&&&&&&&&&&&&&&&&&&&&&&&&&&&&&&&&&&&&&&&&&&&&&&&
%&&&&&&&&&&&&&&&&&&&&&&&&&&&&&&&&&&&&&&&&&&&&&&&&&&&&&&&&&&&&&&&&&&&&&&&&&&&

	In the next step we want to prove an essential property for Gevrey-modulation spaces. Subsequently we always mean algebras 
under pointwise multiplication when speaking of algebras. This property is important in two ways. On the one hand it is needed to 
handle semi-linear problems as (\ref{Gevrey2}) or (\ref{Gevrey4}) with analytic nonlinearity $f(u)$. On the other hand it is 
a useful tool in the proof of the non-analytic superposition result in Section \ref{Gevreysuperposition}. At this point we can 
already mention results of Iwabuchi in \cite{iwabuchi}. However he imposed particular conditions on the integrability parameters. 
%In contrast to that we will state a fundamental algebra result by using the strong Gevrey weights of the spaces $\mathcal{GM}_{p,q}^s(\R^n)$. 
Of some importance for our proof will be the following elementary lemma, see \cite{Bou}, \cite{brs}.

\begin{lem} \label{weightEstimate}
Let $ s > 1$ and put $\delta:= 2-2^{1/s}$.
Then 
\[ e^{|k|^{\frac{1}{s}}} \leq e^{|l|^{\frac{1}{s}}} e^{|l-k|^{\frac{1}{s}}} e^{-\delta \min\{|l-k|, |l|\}^{\frac{1}{s}}}, \]
holds for arbitrary $k,l\in \Z^n$.
\end{lem}
	
After these preparations we can state the main result of this section.

\begin{thm} \label{algebra}
Let $1\leq p_1,p_2,q \leq \infty$ and $s>1$. 
Define $p$ by $\frac{1}{p}:= \frac{1}{p_1}+\frac{1}{p_2}$ and assume that $f\in\mathcal{GM}_{p_1,q}^{s}(\R^n)$ and  $g\in\mathcal{GM}_{p_2,q}^{s}(\R^n)$. 
Then $f\cdot g \in\mathcal{GM}_{p,q}^{s}(\R^n)$ and it holds
		\[ \|f \cdot g\|_{\mathcal{GM}_{p,q}^{s}} \leq C \|f\|_{\mathcal{GM}_{p_1,q}^{s}} \|g\|_{\mathcal{GM}_{p_2,q}^{s}} \]
		with a positive constant $C$ which only depends on the choice of the frequency-uniform 
decomposition, the dimension $n$ and the parameters $s$ and $q$. 		
\end{thm}
	
\begin{proof}
Later on we shall use the same strategy of proof as below in slightly different situations.
For this reason and later use we shall take care of all constants showing up in our estimates below.
\\
We know that $\supp \sigma_k \subset Q_k := \{\xi \in\R^n : -1\leq \xi_i-k_i\leq 1, \, i=1,\ldots,n\}$. 
Further on we introduce the notations $f_j (x)= \F^{-1} (\sigma_j \F f)(x)$ and $g_l (x)= \F^{-1} (\sigma_l \F g)(x)$ 
for $j,l\in\Z^n$. At least formally we have the following representation of the product $f\cdot g$ as
		\[ f\cdot g = \sum_{j,l\in Z^n} f_j \cdot g_l. \]
H\"older's inequality yields
\begin{eqnarray*}
\Big|\sum_{j,l\in Z^n} f_j \cdot g_l\, \Big| & \le & \Big(\sum_{j\in Z^n} \|\, f_j\, \|_{L^\infty}^2\Big)^{1/2} 
\Big(\sum_{l \in \Z^n} \| \, g_l\, \|_{L^\infty}^2 \Big)^{1/2} \\
	& \le & C \, \|\,f\, \|_{\mathcal{GM}_{\infty,\infty}^{s}} \|\,g\, \|_{\mathcal{GM}_{\infty,\infty}^{s}}
\end{eqnarray*}
for any $s>0$ and with a constant $C$ independent of $f$ and $g$. This shows convergence of 
$\sum_{j,l\in \Z^n} f_j \cdot g_l$ in $C_{ub}(\R^n)$, hence in $S'(\R^n)$. In view of Lemma \ref{basic}(iii)
it will be sufficient to prove that the sequence $(\sum_{|k|, |l| < N} f_j\, g_l)_N$ is uniformly bounded in 
$\mathcal{GM}_{p,q}^{s}(\R^n)$.\\
Determining the Fourier support of $f_j \cdot g_l$ we see that
\begin{eqnarray*}  \supp \F (f_j g_l)
& = & \supp (\F f_j \ast \F g_l) \\  & \subset& \{\xi\in \R^n :  j_i+l_i-2 \leq \xi_i \leq j_i+l_i+2, \, i=1,\ldots,n\}. 
\end{eqnarray*}
Hence, the term $\F^{-1} (\sigma_k \F (f_j \cdot g_l))$  vanishes if $\|k - (j+l)\|_\infty \ge  3$. So we obtain

\begin{eqnarray*}
\sigma_k \F (f\cdot g) & = & \sigma_k \F \Big(\sum_{j,l\in\Z^n} f_j\cdot g_l \Big) 
 =   \sigma_k \F \Big(\sum_{\substack{j,l \in \Z^n, \\ k_i-3<j_i+l_i < k_i +3, \\ 
i=1,\ldots, n}} f_j g_l \Big)  \\
& \stackrel{[r=j+l]}{=} & \sum_{\substack{r \in \Z^n, \\ k_i-3<r_i < k_i +3, 
\\ i=1,\ldots, n}}  \sum_{l\in\Z^n}\, 
\sigma_k \F \big( f_{r-l} g_l \big) \, .
\end{eqnarray*}
Hence
\begin{eqnarray*}
\left\|\F^{-1} \big(\sigma_k \F (f\cdot g)\big) \right\|_{L^p} 
& \leq & \sum_{\substack{r \in \Z^n, \\ k_i-3<r_i < k_i +3, \\ i=1,\ldots, n}} \sum_{l\in\Z^n} \| \F^{-1} \big( \sigma_k \F ( f_{r-l} g_l) 
\big)\|_{L^p} \\
& \stackrel{[t=r-k]}{=} & \sum_{\substack{t \in \Z^n, \\ -3<t_i < 3, \\ i=1,\ldots, n}} \sum_{l\in\Z^n} \| \F^{-1} 
\big( \sigma_k \F ( f_{t-(l-k)} g_l) \big)\|_{L^p}.
\end{eqnarray*}
These preparations yield the following norm estimates
\begin{eqnarray*}
 \Big( \sum_{k\in\Z^n} e^{|k|^{\frac{1}{s}}q} && \hspace{-0.7cm} \|\F^{-1} \big(\sigma_k \F (f\cdot g) \big)\|_{L^p}^q \Big)^{\frac{1}{q}} 
\\
& \leq & \Bigg( \sum_{k\in\Z^n} e^{|k|^{\frac{1}{s}}q} \Bigg[ \sum_{\substack{t \in \Z^n, \\ -3<t_i < 3, \\ i=1,\ldots, n}}
 \sum_{l\in\Z^n} \| \F^{-1} \big(\sigma_k \F ( f_{t-(l-k)} g_l) \big)\|_{L^p} \Bigg]^q \Bigg)^{\frac{1}{q}} \\
& \leq & \sum_{\substack{t \in \Z^n, \\ -3<t_i < 3, \\ i=1,\ldots, n}} \left( \sum_{k\in\Z^n} e^{|k|^{\frac{1}{s}}q} 
\left[ \sum_{l\in\Z^n} \| \F^{-1} \big(\sigma_k \F ( f_{t-(l-k)} g_l) \big)\|_{L^p} \right]^q \right)^{\frac{1}{q}} \, .
\end{eqnarray*}
Observe
\begin{eqnarray*}
 \| \F^{-1} \big(\sigma_k \F ( f_{t-(l-k)} g_l) \big)\|_{L^p} & = & (2\pi)^{-n/2} \| (\F^{-1} \sigma_k) *  (f_{t-(l-k)} g_l) \, \|_{L^p}
 \\
 &\le & (2\pi)^{-n/2} \| \F^{-1} \sigma_k \, \|_{L^1} \, \|\,  f_{t-(l-k)} g_l \, \|_{L^p}
\\
& =  & (2\pi)^{-n/2} \| \F^{-1} \sigma_0 \, \|_{L^1} \, \|\,  f_{t-(l-k)} g_l \, \|_{L^p} \, ,
 \end{eqnarray*}
where we used Young's inequality. We put $c_1:= (2\pi)^{-n/2} \| \F^{-1} \sigma_0 \, \|_{L^1}$. This implies
\begin{eqnarray*}
\Big( \sum_{k\in\Z^n} e^{|k|^{\frac{1}{s}}q} && \hspace{-0.7cm} \|\F^{-1} \big(\sigma_k \F (f\cdot g) \big)\|_{L^p}^q \Big)^{\frac{1}{q}} 
\\
& \leq & c_1\,  \sum_{\substack{t \in \Z^n, \\ 
-3<t_i < 3, \\ i=1,\ldots, n}} \left( \sum_{k\in\Z^n} e^{|k|^{\frac{1}{s}}q} \left[ \sum_{l\in\Z^n} 
\|f_{t-(l-k)} g_l\|_{L^p} \right]^q \right)^{\frac{1}{q}}\, .
\end{eqnarray*}
We continue by using H\"older's inequality to get
\begin{eqnarray}\label{ws-03}
 \Big( \sum_{k\in\Z^n} e^{|k|^{\frac{1}{s}}q} && \hspace{-0.7cm} \|\F^{-1} \big(\sigma_k \F (f\cdot g) \big)\|_{L^p}^q \Big)^{\frac{1}{q}} 
\nonumber
 \\
& \leq & c_2 \, \max_{\substack{t \in \Z^n, \\ -3<t_i < 3, \\ i=1,\ldots, n}} \left( \sum_{k\in\Z^n} e^{|k|^{\frac{1}{s}}q} 
\left[ \sum_{l\in\Z^n} \|f_{t-(l-k)}\|_{L^{p_1}} \|g_l\|_{L^{p_2}} \right]^q \right)^{\frac{1}{q}} \hspace{1cm}
\end{eqnarray}
with $c_2:= c_1 \, 5^n$.  Lemma \ref{weightEstimate} yields

\begin{eqnarray*}
 e^{|k|^{\frac{1}{s}}} 
\Big[ \sum_{l\in\Z^n} && \hspace{-0.7cm} \|f_{t-(l-k)}\|_{L^{p_1}}\,  \|g_l\|_{L^{p_2}} \Big] 
\\
& \leq &  \sum_{\substack{l\in\Z^n, \\
 |l|\leq |l-k|}} e^{|l-k|^{\frac{1}{s}}} \|f_{t-(l-k)}\|_{L^{p_1}} e^{|l|^{\frac{1}{s}}} \|g_l\|_{L^{p_2}} e^{-\delta|l|^{\frac{1}{s}}} 
\\
& & \qquad + \sum_{\substack{l\in\Z^n, \\ |l-k|\leq |l|}} e^{|l-k|^{\frac{1}{s}}} \|f_{t-(l-k)}\|_{L^{p_1}} e^{|l|^{\frac{1}{s}}} 
\|g_l\|_{L^{p_2}} e^{-\delta|l-k|^{\frac{1}{s}}}  \,.
\end{eqnarray*}
Both parts of this right-hand side will be  estimated separately. We put 
\begin{eqnarray*}
 S_{1,t,k} & := &  \sum_{\substack{l\in\Z^n, \\
 |l|\leq |l-k|}} e^{|l-k|^{\frac{1}{s}}} \|f_{t-(l-k)}\|_{L^{p_1}} e^{|l|^{\frac{1}{s}}} \|g_l\|_{L^{p_2}} e^{-\delta|l|^{\frac{1}{s}}}, 
\\
S_{2,t,k} & := & \sum_{\substack{l\in\Z^n, \\ |l-k|\leq |l|}} e^{|l-k|^{\frac{1}{s}}} \|f_{t-(l-k)}\|_{L^{p_1}} e^{|l|^{\frac{1}{s}}} 
\|g_l\|_{L^{p_2}} e^{-\delta|l-k|^{\frac{1}{s}}}  \, .
\end{eqnarray*}
With $\frac{1}{q}+\frac{1}{q'}=1$  we find
\begin{eqnarray*}
S_{1,t,k} & \stackrel{[j=l-k]}{=} &  \sum_{\substack{j\in\Z^n, \\ |j+k|\leq |j|}} e^{|j|^{\frac{1}{s}}} \|f_{t-j}\|_{L^{p_1}} 
e^{|j+k|^{\frac{1}{s}}} \|g_{j+k}\|_{L^{p_2}} e^{-\delta|j+k|^{\frac{1}{s}}} 
\\
& \leq &   \Big( \sum_{\substack{j\in\Z^n, \\ |j+k|\leq |j|}} \Big| e^{|j|^{\frac{1}{s}}} 
\|f_{t-j}\|_{L^{p_1}} e^{|j+k|^{\frac{1}{s}}} \|g_{j+k}\|_{L^{p_2}} \Big|^q \Big)^{\frac{1}{q}} \\
			& & \qquad \qquad \qquad \times \quad  \Big( \sum_{\substack{j\in\Z^n, \\ 
|j+k|\leq |j|}} \Big| e^{-\delta|j+k|^{\frac{1}{s}}} \Big|^{q'} \Big)^{\frac{1}{q'}} \, .
\end{eqnarray*}
Since
\begin{equation}\label{ws-25}
\Big( \sum_{\substack{j\in\Z^n, \\ 
|j+k|\leq |j|}} \Big| e^{-\delta|j+k|^{\frac{1}{s}}} \Big|^{q'} \Big)^{\frac{1}{q'}} \le 
\Big( \sum_{m \in\Z^n}  e^{-\delta |m|^{\frac{1}{s}} q'} \Big)^{\frac{1}{q'}} =: c_3
\end{equation}
we conclude
\begin{eqnarray*}
\Big(\sum_{k \in \Z^n} S_{1,t,k}^q\Big)^{1/q} & \le  & c_3\,  \Bigg( \sum_{k\in\Z^n} 
\sum_{\substack{j\in\Z^n, \\ |j+k|\leq |j|}} e^{|j|^{\frac{1}{s}}q} \|f_{t-j}\|_{L^{p_1}}^q 
e^{|j+k|^{\frac{1}{s}}q} \|g_{j+k}\|_{L^{p_2}}^q \Bigg)^{1/q}
\\
& \leq &  c_3 \, \Bigg( \sum_{j\in\Z^n} 
e^{|j|^{\frac{1}{s}}q} \|f_{t-j}\|_{L^{p_1}}^q \sum_{k\in\Z^n} e^{|j+k|^{\frac{1}{s}}q} 
\|g_{j+k}\|_{L^{p_2}}^q \Bigg)^{\frac{1}{q}} 
\, .
\end{eqnarray*}
Because of $|j|^{1/s} \le |j-t|^{1/s} + |t|^{1/s}$ we know
\begin{equation}\label{ws-27} 
\max_{\substack{t \in \Z^n, \\ -3<t_i < 3, \\ i=1,\ldots, n}} \sup_{j\in\Z^n} 
e^{|j|^{\frac{1}{s}}-|t-j|^{\frac{1}{s}}} \leq \max_{\substack{t \in \Z^n, \\ -3<t_i < 3, \\ i=1,\ldots, n}} 
e^{|t|^{\frac{1}{s}}} \leq e^{(2\sqrt{n})^{1/s}} =: c_4 < \infty \, .
\end{equation}
This implies 
\[
\Big(\sum_{k \in \Z^n} S_{1,t,k}^q\Big)^{1/q} \le c_3 \, c_4 \, \|g\|_{\mathcal{GM}^s_{p_2,q}} \|f\|_{\mathcal{GM}^s_{p_1,q}},
\]
where $c_3$, $c_4$ are independent of $f,g$ and $t$.
For the second sum  the estimate
\[
\Big(\sum_{k \in \Z^n} S_{2,t,k}^q\Big)^{1/q} \le c_5\,  \|g\|_{\mathcal{GM}^s_{p_2,q}} \|f\|_{\mathcal{GM}^s_{p_1,q}} 
\]
follows by analogous computations. 
Inserting these estimates into \eqref{ws-03} the claim follows in case $\max(p,q)<\infty$.
Remark that all computations can  be done also by taking the $l^\infty$- and $L^\infty$-norm, respectively. 
The proof is complete.
\end{proof}

\begin{rem}
 \rm
(i) Some basic ideas of the above proof are taken over from Bourdaud \cite{Bou}, see also \cite{brs}.
\\
(ii) Also Wang, Lifeng, Boling \cite{wangExp} considered modulation spaces with an exponential weight. More exactly, they
investigated the scale $E^\lambda_{p,q} (\R^n)$, defined as follows.
A tempered distribution $f$ belongs to $E^\lambda_{p,q} (\R^n)$ if 
\[
\|\, f \, \|_{E^\lambda_{p,q}} :=  \Big(\sum_{k \in \Z^n} 2^{\lambda |k|q}\, \|\, \Box_k f\, \|_{L^p}^q\Big)^{1/q} < \infty\, .
\]
For this scale they proved
\begin{equation}\label{ws-04} 
 \|\, f \cdot g\, \|_{E^{\lambda}_{p,q}} \leq c\,  \| f \|_{E_{p_1,\min(1,q)}^{\lambda}} \, \|g\|_{E^{\lambda}_{p_2,\min(1,q)}} \, ,
 \end{equation}
if $\lambda \ge 0$, $0 < p\le p_1,p_2 \le \infty$, $0 < q \le \infty$, and $\frac{1}{p}= \frac{1}{p_1}+\frac{1}{p_2}$.
Let us mention that this is a result parallel to ours.
In case $\lambda = (\log 2)^{-1}$ the spaces $E^\lambda_{p,q} (\R^n)$ coincide with $\mathcal{GM}_{p,q}^{1}(\R^n)$.
In addition, if $\lambda >0$ then we always have
\[
E^\lambda_{p,q} (\R^n) \hookrightarrow \mathcal{GM}_{p,q}^{s}(\R^n)\, , \qquad s>1\, .
\] 
This makes clear that \eqref{ws-04} represents a borderline case with respect to Theorem \ref{algebra}.
\end{rem}

\begin{cor}\label{algebra2}
Let $1\leq p, q \leq \infty$ and $s>1$. 
Then the  Gevrey-modulation space $\mathcal{GM}_{p,q}^{s}(\R^n)$ is an algebra 
under pointwise multiplication.
\end{cor}
	
\begin{proof}
 As a consequence of Theorem \ref{algebra} and Corollary \ref{einbettung}  we obtain
 \begin{eqnarray}\label{ws-07}
 \|f \cdot g\|_{\mathcal{GM}_{p,q}^{s}} & \leq &  C \, \|f\|_{\mathcal{GM}_{2p,q}^{s}} \, \|g\|_{\mathcal{GM}_{2p,q}^{s}} 
\nonumber 
\\
& \leq &  C_1 \, \|f\|_{\mathcal{GM}_{p,q}^{s}} \, \|g\|_{\mathcal{GM}_{p,q}^{s}}
\end{eqnarray}
with $C_1$ independent of $f$ and $g$.
Hence the claim follows.
 \end{proof}

\begin{rem}
\rm
(i) This time the constant $C_1$ depends on the window $\rho$, $n,q,s$ and $p$.
\\
(ii) Concerning the weighted modulation spaces $M_{p,q}^s(\R^n)$ there are several contributions to the algebra problem. 
We refer to  Feichtinger \cite{feichtinger}, Iwabuchi \cite{iwabuchi} and Sugimoto et all \cite{Sugi}.
\\
(iii) In view of Lemma \ref{gleich}  Corollary \ref{algebra2} extends earlier results, obtained in \cite{brs} for $p=q=2$, to the general case. \\
(iv) The restriction to values of $s>1$ is not a technical one.
In \cite{brs} the authors have shown that $\mathcal{GM}_{2,2}^{1} (\R^n)$ 
is not an algebra with respect to pointwise multiplication. 
\end{rem}

%&&&&&&&&&&&&&&&&&&&&&&&&&&&&&&&&&&&&&&&&&&&&&&&&&&&&&&&&&&&&&&&&&&&&&&&&&&&&&&&&&&&&&&&&&&&&&&&&&&&&&&&&&&&&&&&&&&&&
%&&&&&&&&&&&&&&&&&&&&&&&&&&&&&&&&&&&&&&&&&&&&&&&&&&&&&&&&&&&&&&&&&&&&&&&&&&&&&&&&&&&&&&&&&&&&&&&&&&&&&&&&&&&&&&&&&&&&

	\section{A Non-analytic Superposition Result on Gevrey-modulation Spaces} \label{Gevreysuperposition}		

%&&&&&&&&&&&&&&&&&&&&&&&&&&&&&&&&&&&&&&&&&&&&&&&&&&&&&&&&&&&&&&&&&&&&&&&&&&&&&&&&&&&&&&&&&&&&&&&&&&&&&&&&&&&&&&&&&&&&
%&&&&&&&&&&&&&&&&&&&&&&&&&&&&&&&&&&&&&&&&&&&&&&&&&&&&&&&&&&&&&&&&&&&&&&&&&&&&&&&&&&&&&&&&&&&&&&&&&&&&&&&&&&&&&&&&&&&&

%	The idea of introducing the subsequent superposition operators is given in \cite{brs}. It will turn out that we are 
%able to obtain similar results for Gevrey-modulation spaces. \\
	We need to proceed with some preparations.  An essential tool in proving our main result 
will be a certain subalgebra property of the Gevrey-modulation spaces $\mathcal{GM}_{p,q}^{s}$. Therefore we consider the 
following  decomposition of the phase space. Let $R>0$ and $\epsilon=(\epsilon_1,\ldots, \epsilon_n)$ be fixed with 
$\epsilon_j \in \{0,1\}$, $j=1,\ldots,n$. Then a decomposition of $\R^n$ into $(2^n+1)$ parts is given by
\[ P_R := \{ \xi\in\R^n :\: |\xi_j|\leq R, j=1,\ldots, n \} \]
and
\[ P_R(\epsilon) := \{ \xi\in \R^n: \: \sgn (\xi_j) = (-1)^{\epsilon_j}, \: j=1,\ldots,n \} \setminus P_R. \]
For given $p,q,s$, $\epsilon=(\epsilon_1,\ldots, \epsilon_n)$ and $R>0$ we introduce the spaces
	\[ \mathcal{GM}_{p,q}^{s}(\epsilon, R) := \{ f\in \mathcal{GM}_{p,q}^{s}(\R^n): \: \supp \F(f) \subset P_R(\epsilon) \}\, . \]
As above we will use the convention that for given $q$, $1 \le q \le \infty$, the number $q'$ is defined by 
$\frac{1}{q}+\frac{1}{q'}=1$.

\begin{prop} \label{subalgebra1}
Let $1\leq p, q \leq \infty$, $s>1$ and $R\geq 2$. We put $\frac 1p := \frac{1}{p_1} + \frac{1}{p_2}$. 
Then, for any admissible $\epsilon$ 
\[ 
\|\, f  \cdot g\, \|_{\mathcal{GM}_{p,q}^{s}} \leq D_R \, \|f\|_{\mathcal{GM}_{p_1,q}^{s}}\,  \|g\|_{\mathcal{GM}_{p_2,q}^{s}} 
\]
	holds for all $f \in\mathcal{GM}_{p_1,q}^{s}(\epsilon, R)$ and all $g \in\mathcal{GM}_{p_2,q}^{s}(\epsilon, R)$,  
where the constant $D_R$ is given by
\[ 
D_R := C_0 \, \left(\int_{\delta q' (R-2)^{\frac{1}{s}}}^\infty y^{sn-1} e^{-y} \, dy\right)^{\frac{1}{q'}}. 
\]
Here  $C_0>0$ denotes a constant which  depends only on $p,q,s$ and $n$. 
\end{prop}

\begin{proof}
	Let $f\in\mathcal{GM}_{p_1,q}^{s}(\epsilon, R)$ and $g\in\mathcal{GM}_{p_2,q}^{s}(\epsilon, R)$. By
	\[ \supp (\F f \ast \F g) \subset \{ \xi +\eta: \xi \in \supp \F (f), \eta \in \supp \F (g) \} \]
	we have $\supp \F (fg) \subset P_R(\epsilon)$. In order to show the subalgebra property we follow the same steps as in the
 proof of Theorem \ref{algebra}. 
 Let 
 \[
 P_R^* (\epsilon):= \Big\{k \in \Z^n: \quad \|k\|_\infty>R-1\, , \quad \sgn (k_j) = (-1)^{\epsilon_j}, \: j=1,\ldots,n\Big\}.
 \]
 Hence, if $\supp \sigma_k \cap P_R(\epsilon) \neq \emptyset$, then $k \in P_R^*(\epsilon)$ follows.
In our situation the estimate \eqref{ws-03} can be rewritten as

\begin{align*}
 \Big( \sum_{k\in P_R^* (\epsilon)} e^{|k|^{\frac{1}{s}}q} \|\F^{-1} \big(\sigma_k \F (f\cdot g) \big)\|_{L^p}^q \Big)^{\frac{1}{q}}  \hspace{7cm} 
\end{align*}
\begin{eqnarray*}
& \leq & c_2 \, \max_{\substack{t \in \Z^n, \\ -3<t_i < 3, \\ i=1,\ldots, n}} \Bigg( \sum_{k\in P_R^* (\epsilon)} e^{|k|^{\frac{1}{s}}q} 
\Bigg[ \sum_{\substack{l\in\Z^n, \\ l, t-(l-k)\in P_R^* (\epsilon)}} \|f_{t-(l-k)}\|_{L^{p_1}} \|g_l\|_{L^{p_2}} \Bigg]^q \Bigg)^{\frac{1}{q}}.
\end{eqnarray*}
According to this estimate we introduce the abbreviations 
 \begin{eqnarray*}
 S_{1,t,k} & := &  \sum_{\substack{l\in\Z^n: ~ l, t-(l-k)\in P_R^* (\epsilon), \\
 |l|\leq |l-k|}} e^{|l-k|^{\frac{1}{s}}} \|f_{t-(l-k)}\|_{L^{p_1}} e^{|l|^{\frac{1}{s}}} \|g_l\|_{L^{p_2}} e^{-\delta|l|^{\frac{1}{s}}}, 
\\
S_{2,t,k} & := & \sum_{\substack{l\in\Z^n: ~ l, t-(l-k)\in P_R^* (\epsilon), \\ |l-k|\leq |l|}} e^{|l-k|^{\frac{1}{s}}} \|f_{t-(l-k)}\|_{L^{p_1}} e^{|l|^{\frac{1}{s}}} 
\|g_l\|_{L^{p_2}} e^{-\delta|l-k|^{\frac{1}{s}}}  
\end{eqnarray*}
for all $k \in P_R^* (\epsilon)$. As above we conclude 
\begin{eqnarray*}
S_{1,t,k} & \leq &   \Bigg(  \sum_{\substack{l\in\Z^n, \\ l, t-(l-k)\in P_R^* (\epsilon), \\
 |l|\leq |l-k|}} \Big| e^{|l-k|^{\frac{1}{s}}} \|f_{t-(l-k)}\|_{L^{p_1}} e^{|l|^{\frac{1}{s}}} \|g_l\|_{L^{p_2}}\Big|^q \Bigg)^{\frac{1}{q}} \\
			& & \qquad \qquad \qquad \times \quad  \Bigg( \sum_{\substack{l\in\Z^n: ~ l, t-(l-k)\in P_R^* (\epsilon), \\
 |l|\leq |l-k|}} e^{-\delta q' |l|^{\frac{1}{s}}}   \Bigg)^{\frac{1}{q'}} \, .
\end{eqnarray*}
By definition of the set $P_R^* (0, \ldots \, , 0)$ we find in case $R\geq 2$
\begin{eqnarray*}
 \sum_{\substack{l\in\Z^n, \\  l, t-(l-k)\in P_R^* (0,\,  \ldots \, ,0), \\
 |l|\leq |l-k|}} e^{-\delta q' |l|^{\frac{1}{s}}}   
&\le &  
\sum_{l\in \Z^n:~ \|l\|_\infty > R-1}  e^{-\delta q'|l|^{\frac{1}{s}}}  
\\
&\le &  
\sum_{\substack{l\in\Z^n, \\ \|l\|_\infty > R-1}} \int_{\substack{x\in\R^n, \\ l_i -1 \le x_i\le l_i, \\ i=1, \ldots \, n}} e^{-\delta q'|x|^{\frac{1}{s}}}\, dx  
\\
&\le &   \int_{\|x\|_\infty > R-2} e^{-\delta q'|x|^{\frac{1}{s}}}\, dx 
\\
&\le &  
\int_{|x| >R-2} e^{-\delta q'|x|^{\frac{1}{s}}}\, dx \, .
\end{eqnarray*}
A symmetry argument yields the same estimate in case $\epsilon \neq (0,\,\ldots \, ,0)$.
By means of some simple calculations we obtain
\begin{eqnarray}\label{ws-06}
\int_{|x| >R-2} e^{-\delta q'|x|^{\frac{1}{s}}}\, dx 
&=&    2\, \frac{\pi^{n/2}}{\Gamma (n/2)}\,  \int_{R-2}^\infty r^{n-1} e^{-\delta q' r^{\frac{1}{s}}} \, dr
\nonumber
\\
& \stackrel{[t=r^{\frac{1}{s}}]}{=} & 2\, \frac{\pi^{n/2}}{\Gamma (n/2)}\, s\,  
\int_{(R-2)^{\frac{1}{s}}}^\infty t^{sn-1} e^{-\delta q' t} \, dt 
\nonumber
\\
& \stackrel{[y=\delta q' t]}{=} & 2\, \frac{\pi^{n/2}}{\Gamma (n/2)}\, s\,  (\delta q')^{-sn} 
\int_{\delta q' (R-2)^{\frac{1}{s}}}^\infty y^{sn-1} e^{-y} \, dy
\nonumber
\\
& = : & E_R\, .
\end{eqnarray}
Hence 
\[
\Big(\sum_{k \in P_R^* (\epsilon)} S_{1,t,k}^q\Big)^{1/q} 
 \leq   E_R^{1/q'} \, \Bigg( \sum_{j\in\Z^n} 
e^{|j|^{\frac{1}{s}}q} \|f_{t-j}\|_{L^{p_1}}^q \sum_{k\in\Z^n} e^{|j+k|^{\frac{1}{s}}q} 
\|g_{j+k}\|_{L^{p_2}}^q \Bigg)^{\frac{1}{q}} 
\, .
\]
With $c_4$ defined as above this implies 
\[
\Big(\sum_{k \in \Z^n} S_{1,t,k}^q\Big)^{1/q} \le E_R^{1/q'} \, c_4 \, \|g\|_{\mathcal{GM}^s_{p_2,q}} \|f\|_{\mathcal{GM}^s_{p_1,q}}. 
\]
For the second sum  the estimate
\[
\Big(\sum_{k \in \Z^n} S_{2,t,k}^q\Big)^{1/q} \le E_R^{1/q'} \, c_5\,  \|g\|_{\mathcal{GM}^s_{p_2,q}} \|f\|_{\mathcal{GM}^s_{p_1,q}} 
\]
follows by analogous computations.
\end{proof}

Arguing as in proof of Corollary \ref{algebra2} we obtain the following.

\begin{cor} \label{subalgebra}
	Let $1\leq p, q \leq \infty$ and $s>1$. For $R\geq 2$ and any $\epsilon$ the spaces $\mathcal{GM}_{p,q}^{s}(\epsilon, R)$
	are subalgebras of $\mathcal{GM}_{p,q}^{s}$. Furthermore, it holds
\[ 
\|\, f  \cdot g\,\|_{\mathcal{GM}_{p,q}^{s}} \leq F_R \, \|f\|_{\mathcal{GM}_{p,q}^{s}}\,  \|g\|_{\mathcal{GM}_{p,q}^{s}} 
\]
for all $f,g \in\mathcal{GM}_{p,q}^{s}(\epsilon, R)$. The constant $F_R$ can be specified as
\[ F_R := C_1 \, \left(\int_{\delta q' (R-2)^{\frac{1}{s}}}^\infty y^{sn-1} e^{-y} \, dy\right)^{\frac{1}{q'}}\, , 
\]
	where the constant $C_1>0$ depends only on $p,q,s$ and $n$. 
\end{cor}

Note that in the following we assume every function to be real-valued unless it is explicitly stated that 
complex-valued functions are allowed.

In order to establish the next result we need to recall some lemmas. 
The first one concerns a standard estimate of Fourier multipliers, see, e.g., \cite[Theorem 1.5.2]{triebel}.
By $H^s (\R^n)$ we denote the Sobolev space of fractional order $s$ built on $L_2 (\R^n)$.

\begin{lem} \label{bernstein}
Let $1 \le r \le \infty$ and assume that $s>\frac{n}{2}$. Then there exists a constant $C>0$ such that
\[ 
\|\F^{-1} \left( \phi\F f \right)\|_{L^r} \leq C \|\phi\|_{H^s} \|f\|_{L^r} 
\]
holds for all $f\in L^r(\R^n)$ and all $\phi \in H^s(\R^n)$.
\end{lem}

The next two technical lemmas have been taken from \cite{brs}.

	\begin{lem} \label{lma46brs}
		Let $N\in \N$ and suppose $a_1, a_2,\ldots, a_N$ to be complex numbers. Then it holds
		\[ a_1 \cdot a_2 \cdot \ldots \cdot a_N -1  = \sum_{l=1}^{N} \sum_{\substack{ j=(j_1,\ldots, j_l), \\ 
0\leq j_1 < \ldots < j_l \leq N }} (a_{j_1} -1) \cdot \ldots \cdot (a_{j_l}-1). \]
	\end{lem}
	\begin{proof}
		Cf. Lemma 4.6. in \cite{brs}.
	\end{proof}
	
	\begin{lem} \label{lma45brs}
		Let $\alpha>0$. Define 
		\[ f(t) := \int_{t}^\infty e^{-y} y^{\alpha-1} \, dy, \qquad t\geq 0. \]
		The inverse $g$ of the function $f$ maps $(0,\Gamma(\alpha)]$ onto $[0,\infty)$ and it holds
		\[ \lim_{u\downarrow 0} \frac{g(u)}{\log \frac{1}{u}} = 1. \]
	\end{lem}
	\begin{proof}
		Cf. Lemma 4.5. in \cite{brs}.
	\end{proof}

The non-analytic superposition, which will be stated in Theorem \ref{Superposition}, is based on the following lemma.

\begin{lem} \label{estSuper}
Let $s>1$, $1< p< \infty$ and $1\leq q \le \infty$. Suppose $u \in \mathcal{GM}_{p,q}^s(\R^n)$. Then it holds
\[ \| e^{iu}-1\|_{\mathcal{GM}_{p,q}^s} \leq c \, \|u\|_{\mathcal{GM}_{p,q}^s}\, \left\{
\begin{array}{lll} e^{b \|u\|^{\frac{1}{s}}_{\mathcal{GM}_{p,q}^s}\log 
\|u\|_{\mathcal{GM}_{p,q}^s}} & \quad &  \mbox{ if } \|u\|_{\mathcal{GM}_{p,q}^s} > 1, \\
1 & \quad&   \mbox{ if } \|u\|_{\mathcal{GM}_{p,q}^s} \leq 1
\end{array}\right. 
\]
with constants $b,c>0$ independent of $u$.
\end{lem}

\begin{proof}
This proof basically follows the same steps as  the proof of Theorem 2.3 in \cite{brs}. \\
{\em Step 1.} Let $u\in \mathcal{GM}_{p,q}^s(\R^n)$ satisfying $\supp \F (u) \subset P_R$ for some $R\geq 2$. \\
First  we consider the Taylor expansion
\[ e^{\i u}-1 = \sum_{l=1}^r \frac{(\i u)^l}{l!} + \sum_{l=r+1}^{\infty} \frac{(\i u)^l}{l!} \]
resulting in the norm estimate
\[ 
\| e^{\i u}-1\|_{\mathcal{GM}_{p,q}^s} \leq \Big\| \sum_{l=1}^r \frac{(\i u)^l}{l!} 
\Big\|_{\mathcal{GM}_{p,q}^s} + \Big\| \sum_{l=r+1}^{\infty} \frac{(\i u)^l}{l!} \Big\|_{\mathcal{GM}_{p,q}^s} =: S_1 + S_2. 
\]
By Corollary \ref{algebra2}, see in particular \eqref{ws-07}, we obtain
\[ S_2 \leq \sum_{l=r+1}^\infty \frac{1}{l!} \|u^l\|_{\mathcal{GM}_{p,q}^s} \leq \frac{1}{C_1} 
\sum_{l=r+1}^{\infty} \frac{ (C_1\, \|u\|_{\mathcal{GM}_{p,q}^s})^l}{l!}. 
\]
		Now we choose $r$ as a function of $\|u\|_{\mathcal{GM}_{p,q}^s}$ and distinguish two cases:
		\begin{enumerate}
			\item $C_1\, \|u\|_{\mathcal{GM}_{p,q}^s}>1$. Assume that
\begin{equation}\label{ws-11} 
3\, C_1\, \|u\|_{\mathcal{GM}_{p,q}^s} \leq r \leq 3\, C_1 \, \|u\|_{\mathcal{GM}_{p,q}^s} +1 
\end{equation}
and recall Stirling's formula $l! = \Gamma(l+1) \geq l^l e^{-l} \sqrt{2\pi l}$. Thus, we get
\begin{eqnarray*}
\sum_{l=r+1}^{\infty} \frac{(C_1\|u\|_{\mathcal{GM}_{p,q}^s})^l}{l!} & \leq &  
\sum_{l=r+1}^{\infty} \left( \frac{r}{l} \right)^l \left(\frac{e}{3}\right)^l \frac{1}{\sqrt{2\pi l}} 
\\
&\leq &   \sum_{l=r+1}^\infty \left(\frac{e}{3} \right)^l \leq  \frac{3}{3-e}.
\end{eqnarray*}
			\item $C_1 \, \|u\|_{\mathcal{GM}_{p,q}^s}\leq 1$. It follows
\begin{equation*}
 \sum_{l=r+1}^{\infty} \frac{(C_1\, \|u\|_{\mathcal{GM}_{p,q}^s})^l}{l!} \leq  
\sum_{l=1}^{\infty} \frac{(C_1\, \|u\|_{\mathcal{GM}_{p,q}^s})^l}{l!}
\leq  C_1\,  e\,  \|u\|_{\mathcal{GM}_{p,q}^s}.
\end{equation*}
\end{enumerate}

Both together can be summarized as 
\begin{equation}\label{ws-08}
 S_2 \le C_2 \, \|u\|_{\mathcal{GM}_{p,q}^s}\, , \qquad C_2 := \max \Big(C_1 e, \frac{3}{3-e}\Big).
\end{equation}
To estimate $S_1$ we check the support of $\F u^\ell $ and find
\begin{eqnarray*}
S_1 = \Big\| \sum_{l=1}^r \frac{(\i u)^l}{l!} \Big\|_{\mathcal{GM}_{p,q}^{s}} & = & 
\Big( \sum_{k\in\Z^n} e^{|k|^{\frac{1}{s}}q} \Big\|\Box_k \Big(\sum_{l=1}^{r} \frac{(\i u)^l}{l!} \Big) 
\Big\|_{L^p}^q \Big)^{\frac{1}{q}} \\
& = & \Big( \sum_{\substack{k\in\Z^n, \\ -Rr-1<k_i< Rr+1, \\ i=1,\ldots,n}} e^{|k|^{\frac{1}{s}}q} 
\Big\|\Box_k \Big(\sum_{l=1}^{r} \frac{(\i u)^l}{l!} \Big) \Big\|_{L^p}^q \Big)^{\frac{1}{q}} \\
& \leq & \Big( \sum_{\substack{k\in\Z^n, \\ -Rr-1<k_i< Rr+1, \\ i=1,\ldots,n}} e^{|k|^{\frac{1}{s}}q} 
\|\Box_k (e^{\i u}-1) \|_{L^p}^q \Big)^{\frac{1}{q}} + S_2\, .
	\end{eqnarray*}
Concerning  $S_2$ we proceed as above. 
To estimate the first part we observe that
\[
C_3 := \sup_{k \in \Z^n } \, \| \, \sigma_k \, \|_{H^t} = \| \, \sigma_0 \, \|_{H^t} <\infty\, , 
\]
see Lemma \ref{bernstein}.
Furthermore,  
$\cos, \sin $ are Lipschitz continuous and consequently we get
\begin{eqnarray*}
\|\Box_k (e^{\i u}-1) \|_{L^p} & \le &  C_3 \, \|e^{\i u}-1 \|_{L^p} \\
	& \le &   C_3 \, (\|\cos  u - \cos 0 \|_{L^p} + \|\sin  u - \sin 0 \|_{L^p})
\\
& \le &   2\, C_3 \, \| u - 0 \|_{L^p} \, .
\end{eqnarray*}
This implies
\begin{align*}
\Big( \sum_{\substack{k\in\Z^n, \\ -Rr-1<k_i< Rr+1, \\ i=1,\ldots,n}} e^{|k|^{\frac{1}{s}}q} 
\| \Box_k (e^{\i u}-1) \|_{L^p}^q \Big)^{\frac{1}{q}} \hspace{7cm}
\end{align*}
\begin{eqnarray*}
& \le & 2\, C_3 \, \|\,  u \, \|_{L^p}\,  \Big( \sum_{\substack{k\in\Z^n, \\ -Rr-1<k_i< Rr+1, \\ i=1,\ldots,n}} e^{|k|^{\frac{1}{s}}q} 
\Big)^{\frac{1}{q}} \, .
\end{eqnarray*}
With a calculation similar to \eqref{ws-06}
we derive
\begin{eqnarray}\label{ws-10}
\sum_{\substack{k\in\Z^n, \\ -Rr-1<k_i< Rr+1, \\ i=1,\ldots,n}} e^{|k|^{\frac{1}{s}}q} 
& \le & \int_{\|\, x\, \|_\infty < Rr+1}\, e^{|x|^{\frac{1}{s}}q}  \, dx
\nonumber
\\
& \le & \int_{|x|< \sqrt{n}(Rr+1)}\, e^{|x|^{\frac{1}{s}}q}  \, dx
\nonumber
\\
&=&    2\, \frac{\pi^{n/2}}{\Gamma (n/2)}\,  \int_{0}^{\sqrt{n} (Rr+1)} \tau^{n-1} e^{ \tau^{\frac{1}{s}} q} \, d\tau
\nonumber
\\
& \le & 2\, \frac{\pi^{n/2}}{\Gamma (n/2)}\, \frac sq\,  (\sqrt{n}(Rr+1))^{n-\frac{1}{s}}\, 
e^{(\sqrt{n}(Rr+1))^{\frac{1}{s}}q} \, . \hspace{1cm}
\end{eqnarray}
To simplify notation we define
\[
C_4 := \sup_{r \in \N} \, \sup_{R \geq 2}\, \Big(2\, \frac{\pi^{n/2}}{\Gamma (n/2)}\, \frac sq\,  n^{(n-\frac{1}{s})/2}\,
 e^{-(\sqrt{n}(Rr+1))^{\frac{1}{s}}q}\Big)^{1/q}\, .
\]
In addition we shall use
\[
\| u \|_{L^p}\le C_5 \, \|u\|_{\mathcal{GM}_{p,q}^{s}} \, , \qquad C_5 := \Big(\sum_{k\in \Z^n} e^{-|k|^{1/s}q'}\Big)^{1/q'}
\]
which follows from H\"older's inequality.
Summarizing we have found
\[
\| e^{\i u}-1\|_{\mathcal{GM}_{p,q}^s} \le \Big(2 \, C_2 + 2\, C_5\, C_4\,C_3\,  
 e^{2(\sqrt{n} (Rr+1))^{\frac{1}{s}}}\Big)\,  \|u\|_{\mathcal{GM}_{p,q}^{s}},
\]
where the chosen constants depend on the dimension $n$, the weight parameter $s$ and the integrability parameters $p$ and $q$. 
Next we apply \eqref{ws-11} which results in
\begin{equation} \label{expest}
\| e^{\i u}-1 \|_{\mathcal{GM}_{p,q}^{s}} \leq c_0 
\|u\|_{\mathcal{GM}_{p,q}^{s}} \left( 1 + e^{b_0 R^{\frac{1}{s}} \| u\|_{\mathcal{GM}_{p,q}^{s}}^{\frac{1}{s}} } \right)\, , 
\end{equation}
valid for all $u\in \mathcal{GM}_{p,q}^{s}(\R^n)$ satisfying $\supp \F(u) \subset P_R$ and with positive constants $b_0,c_0$ depending on $n,s$ 
and $q$ but independent of $u, r$ and $R$. 
\\
{\em Step 2.}	The next step consists of choosing general $u\in \mathcal{GM}_{p,q}^{s}(\R^n)$. 
Here we need the restriction  $1< p< \infty$. For those $p$ the 
characteristic functions $\chi$ of cubes are Fourier multipliers in $L^p(\R^n)$ by the famous Riesz Theorem and therefore also in 
$\mathcal{GM}_{p,q}^{s}(\R^n)$.
In addition we shall make use of the fact that the norm of the operator
$f \mapsto \F^{-1} \chi\,  \F f$ does not depend on the size of the cube.
Below we shall denote this norm by $C_6 = C_6 (p)$. 
We refer to Lizorkin \cite{Li} for all details.
For decomposing $u$ on the phase 
space we introduce functions $\chi_{R,\epsilon}$ and $\chi_R$, that is, the characteristic functions of the sets $P_R(\epsilon)$ and $P_R$, 
respectively. By defining
	\begin{eqnarray*}
		u_\epsilon (x) & = & \F^{-1} [\chi_{R,\epsilon} (\xi) (\F u)(\xi) ] (x) , \\
		u_0 (x) & = & \F^{-1} [\chi_R (\xi) (\F u) (\xi) ] (x)
	\end{eqnarray*}
	we can rewrite $u$ as
	\begin{equation} \label{urepr}
		u(x) = u_0(x) + \sum_{\epsilon\in I} u_\epsilon (x) ,
	\end{equation}
	where $I$ is the set of all $\epsilon=(\epsilon_1,\ldots, \epsilon_n)$ with $\epsilon_j \in \{0,1\}$, $j=1,\ldots,n$. 
Hence
\begin{equation} \label{normrepr}
\| u\|_{\mathcal{GM}_{p,q}^{s}} \le  \| u_0\|_{\mathcal{GM}_{p,q}^{s}}  + 
\sum_{\epsilon\in I} \| u_\epsilon \|_{\mathcal{GM}_{p,q}^{s}}
\end{equation}
and 
\[
\max \Big(\| u_0\|_{\mathcal{GM}_{p,q}^{s}},  \| u_\epsilon \|_{\mathcal{GM}_{p,q}^{s}}\Big) \le C_6 \, \| \, u\,  \|_{\mathcal{GM}_{p,q}^{s}}\, .
\] 
 Due to the  representation (\ref{urepr}) and using an appropriate enumeration Lemma \ref{lma46brs} leads to
\[ 
e^{\i u} -1 = \sum_{l=1}^{2^n +1} \sum_{0\leq j_1 <\ldots < j_l \leq 2^n} 
(e^{\i u_{j_1}} -1)\cdot \ldots \cdot (e^{\i u_{j_l}} -1) \, .
\]
Corollary  \ref{algebra2} immediately yields
\begin{equation}\label{ws-12}
\|e^{\i u} -1\|_{\mathcal{GM}_{p,q}^{s}} \leq  \sum_{l=1}^{2^n +1} C^{l-1}_1\, 
\sum_{0\leq j_1 <\ldots < j_l \leq 2^n} \|e^{\i u_{j_1}} -1 \|_{\mathcal{GM}_{p,q}^{s}} 
\cdot \ldots \cdot \| e^{\i u_{j_l}} -1 \|_{\mathcal{GM}_{p,q}^{s}}. 
\end{equation}
By Corollary \ref{subalgebra}, \eqref{normrepr} and \eqref{expest} it follows
\begin{eqnarray}
\| e^{\i u_{j_k}} -1\|_{\mathcal{GM}_{p,q}^{s}} & = & \Big\|\sum_{l=1}^\infty \frac{(\i u_{j_k})^l}{l!} \, \Big\|_{\mathcal{GM}_{p,q}^{s}}
\le 	\frac{1}{F_R} \Big(e^{F_R\,  \|u_{j_k}\|_{\mathcal{GM}_{p,q}^{s}}} -1 \Big) 
\nonumber 
\\
& \leq & \frac{1}{F_R} \Big( e^{F_R\,C_6\,  \|u\|_{\mathcal{GM}_{p,q}^{s}}} -1 \Big), \label{est1} 
\end{eqnarray}
as well as
\begin{eqnarray}
\| e^{\i u_0} -1\|_{\mathcal{GM}_{p,q}^{s}} & \leq &  c_0 \, C_6\,  \|u\|_{\mathcal{GM}_{p,q}^{s}} 
\Big( 1 + e^{b_0 \, C_6^{\frac{1}{s}}\, R^{\frac{1}{s}} \| u\|_{\mathcal{GM}_{p,q}^{s}}^{\frac{1}{s}} } \Big)\, ,  \label{est2}
\end{eqnarray}
where we used the Fourier multiplier assertion mentioned at the beginning of this substep.
The final step in our  proof is to choose the number $R$ 
as a function of $\|u\|_{\mathcal{GM}_{p,q}^{s}}$ such that \eqref{est1} and \eqref{est2} will be  approximately 
of the same size. 
\\
{\em Substep 2.1.} Let  $\|u\|_{\mathcal{GM}_{p,q}^{s}} \leq 1$. We choose $R=3$. 
Then \eqref{ws-12} combined with \eqref{est1} and \eqref{est2} results in the estimate 
\begin{equation}\label{ws-13}
 \|e^{\i u} -1\|_{\mathcal{GM}_{p,q}^{s}} \leq C_7 \, \|u\|_{\mathcal{GM}_{p,q}^{s}} ,
\end{equation}
where $C_7$ does not depend on $u$.\\
{\em Substep 2.2.} Let $\|u\|_{\mathcal{GM}_{p,q}^{s}} >1$. 
As mentioned in Corollary \ref{subalgebra} we know that the algebra constant $F_R$ in \eqref{est1} 
is a function of $R$, i.e.,
\[
 F_R := C_1 \, \left(\int_{\delta q' (R-2)^{\frac{1}{s}}}^\infty y^{sn-1} e^{-y} \, dy\right)^{\frac{1}{q'}}\, .
\]
Taking into account that
	\begin{itemize}
		\item $F_R$ (as a function of $R$) is strictly decreasing and positive and
		\item $\displaystyle \lim_{R\to \infty} F_R = 0$
	\end{itemize}
we can easily set 
\[ 
\frac{F_R}{F_2}  = \|u\|_{\mathcal{GM}_{p,q}^{s}}^{\frac{1}{s}-1}
\]
for some $R>2$.
In view of Lemma \ref{lma45brs} this gives
\[ C_8 \, \Big( f(\delta q' (R-2)^{\frac{1}{s}}) \Big)^{\frac{1}{q'}} = \|u\|_{\mathcal{GM}_{p,q}^{s}}^{\frac{1}{s}-1} \]
with an appropriate positive constant $C_8$.
Thus, by Lemma \ref{lma45brs} it follows
\[ R =  2+ (\delta q')^{-s} \left( g \left(\frac{\|u\|_{\mathcal{GM}_{p,q}^{s}}^{q'(\frac{1}{s}-1)}}{C_8^{q'}} \right) \right)^s \]
and, moreover,
\[ R \leq C_9 + C_{10} \,\Big(q'-\frac{q'}{s} \Big)^s\,  \log^s \|u\|_{\mathcal{GM}_{p,q}^{s}} \, . \]
Note that the constants $C_9$ and $C_{10}$ are independent of $u$. 
Now \eqref{ws-12} combined with \eqref{est1} and \eqref{est2} results in  
\begin{eqnarray}\label{ws-14}
&& \|e^{\i u} -1\|_{\mathcal{GM}_{p,q}^{s}}
\nonumber 
\\
& \leq &  C_{11}\, \max_{\alpha \in \{0,1\}, \beta \in \{0, \ldots \, 2^n\}}
\, \bigg(
c_0 \, C_6\,  \|u\|_{\mathcal{GM}_{p,q}^{s}} 
\Big( 1 + e^{b_0 \, C_6^{\frac{1}{s}}\, R^{\frac{1}{s}} \| u\|_{\mathcal{GM}_{p,q}^{s}}^{\frac{1}{s}} } \Big) \bigg)^\alpha \nonumber
\\ && \qquad \qquad \qquad \times \, \Big( \frac{e^{F_R\,C_6\,  \|u\|_{\mathcal{GM}_{p,q}^{s}}} -1}{F_R} \Big)^\beta
\nonumber
\\
& \le &  C_{12}\, \max_{\alpha \in \{0,1\}, \beta \in \{0, \ldots \, 2^n\}}
\, \bigg(\|u\|_{\mathcal{GM}_{p,q}^{s}} 
\Big( 1 + e^{b_1 \,  \| u\|_{\mathcal{GM}_{p,q}^{s}}^{\frac{1}{s}} \, \log \| u\|_{\mathcal{GM}_{p,q}^{s}}} \Big) \bigg)^\alpha \nonumber
\\ && \qquad \qquad \qquad \times \, \| u\|_{\mathcal{GM}_{p,q}^{s}}^{\beta (1-\frac{1}{s})} \Big( \frac{e^{F_2\,C_6\,  \|u\|^{\frac{1}{s}}_{\mathcal{GM}_{p,q}^{s}}} -1}{F_2} \Big)^\beta
\nonumber
\\
&\le & C_{13}\, \|u\|_{\mathcal{GM}_{p,q}^{s}} \, \Big( 1 + 
e^{b \,  \| u\|_{\mathcal{GM}_{p,q}^{s}}^{\frac{1}{s}} \, \log \| u\|_{\mathcal{GM}_{p,q}^{s}}} \Big)   
\end{eqnarray}
with a constant $C_{13}$ independent of $u$.
\end{proof}

\begin{rem}
\rm
The restriction of $p$ to the interval $(1,\infty)$ is caused by  our decomposition technique, see Step 2 of the preceeding proof.
We do not know whether Lemma \ref{estSuper} extends to $p=1$ and/or $p = \infty$.
\end{rem}

\begin{lem} \label{ContExp}
Let $s>1$, $1< p< \infty$ and $1\leq q \le \infty$.
\\
{\rm (i)} The mapping $u \mapsto e^{iu}-1$
is locally Lipschitz continuous (considered as a mapping of $\mathcal{GM}_{p,q}^s(\R^n)$
into $\mathcal{GM}_{p,q}^s(\R^n)$). 
\\
{\rm (ii)} Assume $u\in \mathcal{GM}_{p,q}^s(\R^n)$ to be fixed and define a function 
$g: \R \mapsto \mathcal{GM}_{p,q}^s(\R^n)$ by $g(\xi) = e^{\i u(x) \xi}-1$. Then the function $g$ is continuous.
\end{lem}

\begin{proof}
Local Lipschitz continuity follows from the identity
\begin{equation}\label{ws-15}
e^{iu}- e^{iv} = (e^{iv}-1)\, (e^{i(u-v)}-1) + (e^{i(u-v)}-1)\, ,
\end{equation}
the algebra property of $\mathcal{GM}_{p,q}^s(\R^n)$ and Lemma \ref{estSuper}.
\\
To prove the continuity of $g$ we also employ the identity \eqref{ws-15}.
The claim follows by using the algebra property and Lemma \ref{estSuper}.
\end{proof}
	
Now we can establish the announced superposition result.

	\begin{thm} \label{Superposition}
		Let  $s>1$, $1< p < \infty$, $1\le q \le  \infty$ and $\mu$ be a complex measure on $\R$ such that
		\begin{equation} \label{FourierEst}
			L_1(\lambda) = \int_{\R} e^{\lambda (|\xi|^{\frac{1}{s}} \log |\xi| )} \, d|\mu| (\xi) < \infty
		\end{equation}
		for any $\lambda >0 $ and such that $\mu(\R) = 0$. \\
		Furthermore, assume that the function $f$ is the inverse Fourier transform of $\mu$. 
Then $f\in C^\infty$ and the composition operator $T_f: u \mapsto f \circ u$ maps $\mathcal{GM}_{p,q}^s$ into $\mathcal{GM}_{p,q}^s$.
	\end{thm}
	\begin{proof}
		Equation \eqref{FourierEst} yields $\int_{\R} d|\mu|(\xi) < \infty$. Thus, $\mu$ is a finite measure and $\mu(\R) =0$ 
makes sense. Now we define the inverse Fourier transform of $\mu$
		\[ f(t) = \frac{1}{\sqrt{2\pi}} \int_{\R} e^{\i \xi t} \, d \mu(\xi). \]
		Moreover, $\int_{\R} |(\i \xi)^j| \, d|\mu|(\xi) <\infty$ is deduced from equation \eqref{FourierEst} 
for all $j\in\N$. This gives $f\in C^\infty$ and due to $\mu(\R)=0$ we can also write $f$ as follows:
		\[ f(t) = \frac{1}{\sqrt{2\pi}} \int_{\R} (e^{\i \xi t}-1) \, d \mu(\xi). \]
		Since $\mu$ is a complex measure we can split it up into real part $\mu_r$ and imaginary part $\mu_i$, 
where each of them is a signed measure. Without loss of generality we proceed our computations only with the positive 
real measure $\mu_r^+$. For all measurable sets $E$ we have $\mu_r^+(E) \leq |\mu|(E)$. \\
		Let $u\in \mathcal{GM}_{p,q}^s(\R^n)$ and define the function $g(\xi) = e^{\i u(x) \xi} -1$ analogously to 
Lemma \ref{ContExp}. Then $g$ is Bochner integrable because of its continuity and taking into account that the measure $\mu_r^+$ is finite. 
Therefore we obtain the Bochner integral
		\[ \int_{\R} \big(e^{\i u(x) \xi} -1 \big)\, d\mu_r^+(\xi) = \int_{\R} g(\xi) \, d\mu_r^+(\xi) \]
		with values in $\mathcal{GM}_{p,q}^s(\R^n)$. By applying Minkowski inequality it follows
		\[ \Big\| \int_{\R} \big(e^{\i u(\cdot) \xi} -1 \big)\, d\mu_r^+(\xi) \Big\|_{\mathcal{GM}_{p,q}^s} 
\leq \int_{\R} \big\| e^{\i u(\cdot) \xi} -1 \big\|_{\mathcal{GM}_{p,q}^s} \, d|\mu|(\xi). \]
Using the abbreviation $\| u\|:= \|u\|_{\mathcal{GM}_{p,q}^s}$,	Lemma \ref{estSuper} together with equation \eqref{FourierEst} yields
		\begin{eqnarray*}
			\int_{|\xi| \|u\|\geq 1} \big\| e^{\i u(\cdot) \xi} -1 \big\|_{\mathcal{GM}_{p,q}^s} \, d|\mu|(\xi) & \leq & 
c\,  \int_{|\xi| \|u\|\geq 1} e^{b \|\xi u\|^{\frac{1}{s}}_{\mathcal{GM}_{p,q}^s}\log \|\xi u\|_{\mathcal{GM}_{p,q}^s} } \, d|\mu|(\xi) \\
				& < & \infty.
		\end{eqnarray*}
In this way also the remaining part of the integral $|\xi| \le 1/ \|u\|$ can be treated. \\
		The same estimates also hold for the measures $\mu_r^-$, $\mu_i^+$ and $\mu_i^-$. Thus, the result is obtained by
\begin{eqnarray*}
\|\sqrt{2\pi} f(u(x))\|_{\mathcal{GM}_{p,q}^s} & = & 
\Big\| \int_{\R} g(\xi) \, d\mu_r^+ - \int_{\R} g(\xi) \, d\mu_r^- \\
& & \qquad + \i \int_{\R} g(\xi) \, d\mu_i^+ - \i \int_{\R} g(\xi) \, d\mu_i^- \Big\|_{\mathcal{GM}_{p,q}^s} 
\\
& \leq &  \int_{\R} \| g(\xi) \|_{\mathcal{GM}_{p,q}^s} \, d|\mu_r^+|  + \int_{\R} \| g(\xi) \|_{\mathcal{GM}_{p,q}^s}\, d|\mu_r^-| \\
& & \qquad + \int_{\R} \| g(\xi)\|_{\mathcal{GM}_{p,q}^s} \, d|\mu_i^+| + \int_{\R} \| g(\xi) \|_{\mathcal{GM}_{p,q}^s} \, d|\mu_i^-|
\, ,
\end{eqnarray*}
where every integral on the right-hand side is finite. Thus, the statement is proved.
\end{proof}

\begin{rem}
 \rm
Theorem \ref{Superposition} is an extension of earlier results obtained for ${\mathcal{GM}^s_{2,2}}(\R^n)$ in \cite{brs}.
\end{rem}

For practical reasons we remark the following consequence.
\begin{cor} \label{SuperpositionCor}
Let the weight parameter $s>1$, $1 <p < \infty$, $1 \leq q \leq \infty$ and $\mu$ be a complex 
measure on $\R$ with the corresponding bounded density function $g$, i.e., $d\mu (\xi) = g(\xi) \, d\xi$. Suppose that
\begin{equation} \label{densCond}
\lim_{|\xi|\to \infty} \frac{|\xi|^{\frac{1}{s}} \log |\xi| }{\log |g(\xi)|} = 0
\end{equation}
and $\displaystyle \int_{\R} d\mu (\xi) = \int_{\R} g(\xi) \, d\xi = 0$. Assume the function $f$ to be the 
inverse Fourier transform of $g$. Then $f\in C^\infty$ and the composition operator $T_f: u \mapsto f \circ u$ 
maps $\mathcal{GM}_{p,q}^s$ into $\mathcal{GM}_{p,q}^s$.
\end{cor}

\begin{proof}
		Most of the work has been done in the proof of Theorem \ref{Superposition}. 
		The condition \eqref{densCond} yields that the modulus of $\displaystyle \lim_{|\xi|\to\infty} \log |g(\xi)|$ 
needs to be infinity. This fact together with the boundedness of $g$ implies $\displaystyle \lim_{|\xi|\to\infty} g(\xi)=0$. 
Moreover, for any $\lambda >0$  there exists a number $N>0$ such that
		\[ - \frac{|\xi|^{\frac{1}{s}} \log |\xi| }{\log |g(\xi)|} \leq \frac{1}{2\lambda} \,, \qquad |\xi|>N\,  .\]
		Thus, we obtain
		\[ |g(\xi)| \leq e^{-2\lambda |\xi|^{\frac{1}{s}} \log |\xi| } \,, \qquad |\xi|>N\, ,\]
		and it follows
		\begin{eqnarray*}
\int_{|\xi|>N} e^{\lambda(|\xi|^{\frac{1}{s}} \log |\xi| )} \, d|\mu|(\xi) & = &  
\int_{|\xi|>N} e^{\lambda(|\xi|^{\frac{1}{s}} \log |\xi| )} |g(\xi)| \, d\xi 
\\
&&  \leq  \int_{|\xi|>N} e^{-\lambda(|\xi|^{\frac{1}{s}} \log |\xi|)} \, d\xi 	<  \infty.
		\end{eqnarray*}
		This completes the proof.
	\end{proof}

\subsection*{One  Example} 

We recall a construction considered in Rodino \cite[Example 1.4.9]{Ro}.
Let $\mu <0$ and let for $t \in \R$
\[
\psi_\mu (t):= \left\{ \begin{array}{lll}
e^{-t^\mu} & \qquad & \mbox{if} \quad t>0\, ;
\\
0 && \mbox{otherwise}\, .
                   \end{array} \right.
\]
By taking
\[
\phi_\mu (t):= \psi_\mu (1-t)\, \cdot \psi_\mu (t)\, , \qquad t\in \R\, ,
\]
we obtain a compactly supported $C^\infty$ function on $\R$.
It follows $\varphi_\mu \in G^s (\R)$ for $s= 1-1/\mu$. Here the classes $G^s(\R)$ refer to classical Gevrey regularity, see \cite[Definition 1.4.1]{Ro}.
We skip the definition here and recall a further result, see \cite[Thm. 1.6.1]{Ro}.
Since $\varphi_\mu \in G^s (\R)$  has compact support there exists a positive constant $c$ and some $\varepsilon >0$ such that
\[
|\F \varphi_\mu (\xi)| \le c \, e^{-\varepsilon \, |\xi|^{1/s}}\, , \quad \xi \in \R\, .
\]
Because of $\phi_\mu (0) =0$ Cor.  \ref{SuperpositionCor} yields the following.

\begin{cor} \label{ex1}
Let the weight parameter $s>1$, $1 <p < \infty$ and  $1 \leq q \leq \infty$.  Suppose that $\mu < \frac 1{1-s}$. 
 Then  $T_{\varphi_\mu}: u \mapsto \varphi_\mu \circ u$ maps $\mathcal{GM}_{p,q}^s$ into $\mathcal{GM}_{p,q}^s$.
\end{cor}

%&&&&&&&&&&&&&&&&&&&&&&&&&&&&&&&&&&&&&&&&&&&&&&&&&&&&&&&&&&&&&&&&&&&&&&&&&&&&&&&&&&&&&&&&&&&&&&&&&&&&&&&&&&&&&&&&&&&&&&&&&&&&&&&&&&&&&&
%&&&&&&&&&&&&&&&&&&&&&&&&&&&&&&&&&&&&&&&&&&&&&&&&&&&&&&&&&&&&&&&&&&&&&&&&&&&&&&&&&&&&&&&&&&&&&&&&&&&&&&&&&&&&&&&&&&&&&&&&&&&&&&&&&&&&&&

	\section{A Non-analytic Superposition Result on Special Modulation Spaces} 
\label{Ultradifferentiablesuperposition}

%&&&&&&&&&&&&&&&&&&&&&&&&&&&&&&&&&&&&&&&&&&&&&&&&&&&&&&&&&&&&&&&&&&&&&&&&&&&&&&&&&&&&&&&&&&&&&&&&&&&&&&&&&&&&&&&&&&&&&&&&&&&&&&&&&&&&&&
%&&&&&&&&&&&&&&&&&&&&&&&&&&&&&&&&&&&&&&&&&&&&&&&&&&&&&&&&&&&&&&&&&&&&&&&&&&&&&&&&&&&&&&&&&&&&&&&&&&&&&&&&&&&&&&&&&&&&&&&&&&&&&&&&&&&&&&

In the previous sections the results in \cite{brs} gave the motivation to use weights of Gevrey type. Now we want to leave the Gevrey frame and 
approach weights of Sobolev type. For brevity we put
\[
\langle x \rangle_{*}:= (e^{2e} + |x|^2)^{1/2}\, , \qquad x \in \R^n\, . 
\]
Let 
\[
w(x):= e^{\log \langle x \rangle_* \, \log\log \langle x \rangle_*}\, , \qquad x \in \R^n\, . 
\]
By means of our normalization we have $\log\log \langle x \rangle_* \ge 1$.
Clearly, in case $t>0$ and $s>1$, there exist positive constants $A_t$ and $B_s$ such that
\[
A_{t}\, \langle x \rangle^t_* \le w(x) \le B_s\,  e^{|x|^{1/s}}\, , \qquad x \in \R^n\, . 
\]
The spaces defined by such a weight may serve as a prototype for weighted modulation spaces  where the weight 
is increasing stronger than any polynomial but weaker than $e^{|x|^{1/s}}$.

\begin{defn}
		The integrability parameters are given by $1\leq p,q \leq \infty$. Let $\rho\in C_0^\infty (\R^n)$ be a 
fixed window. Then the modulation space $\mathcal{U M}^{p,q}(\R^n)$ is the collection of all $f \in L^p (\R^n)$ such that
\[ \|f\|_{\mathcal{U M}^{p,q}} := \Big( \sum_{k\in\Z^n} e^{q \log \langle k \rangle_* \log\log \langle k \rangle_*}
\|\F^{-1} (\Box_k \F f)\|_{L^p}^q \Big)^{\frac{1}{q}} < \infty  \]
		with obvious modifications when $p=\infty$ and/or $q=\infty$.
	\end{defn}
	\begin{rem}
		In fact, the modulation spaces $\mathcal{U M}^{p,q}(\R^n)$ contain all Gevrey-modulation 
spaces $\mathcal{GM}_{p,q}^s(\R^n)$ but are contained in every classical modulation space $M_{p,q}^s(\R^n)$, i.e.
		\[ \bigcup_{s>1} \mathcal{GM}_{p,q}^s(\R^n) \subset \quad  \mathcal{U M}^{p,q}(\R^n) \quad \subset \bigcap_{s\in\R} M_{p,q}^s(\R^n). \]
	\end{rem}

By the same arguments as in Lemma \ref{basic} one can prove the following statements.

\begin{lem}\label{basic2}
(i) The modulation space $\mathcal{UM}^{p,q}(\R^n)$ is a Banach space.
\\
(ii) $\mathcal{UM}^{p,q}(\R^n)$ is independent of the choice of the window $\rho\in C_0^\infty (\R^n)$ in the sense of equivalent norms.
\\
(iii) $\mathcal{UM}^{p,q}(\R^n)$ has the Fatou property, i.e., if $(f_m)_{m=1}^\infty \subset \mathcal{UM}^{p,q}(\R^n)$ 
is a sequence such that 
$ f_m \rightharpoonup f $ (weak convergence in $S'(\R^n)$) and 
\[
 \sup_{m \in \N}\,  \|\, f_m \, \|_{\mathcal{UM}^{p,q}} < \infty\, , 
\]
then $f \in \mathcal{UM}^{p,q}(\R^n)$ follows and 
\[
\|\, f \, \|_{\mathcal{UM}^{p,q}} \le \sup_{m\in \N}\,  \|\, f_m \, \|_{\mathcal{UM}^{p,q}} < \infty\, .
\]
\end{lem}

In order to prove  that $\mathcal{U M}^{p,q}$ is an algebra under pointwise multiplication we need a counterpart of Lemma \ref{weightEstimate}.
Therefore we start with some elementary analysis.
Let 
\[
w_*(t):= \log \langle t \rangle_* \, \log\log \langle t \rangle_*\, , \qquad t \in [0,\infty)\, . 
\]
This function is strongly increasing and its range is given by $[e,\infty)$. 
To proceed as in Section \ref{Gevreysuperposition} we need to prove a counterpart of Lemma \ref{weightEstimate}.
%Furthermore, the function $w_*''$ has only one sign change in $(0,\infty)$. \\

\begin{lem}\label{numer}
There exists a positive real number $s\in (0,1)$ such that
\begin{equation}\label{ws-20}
w_* (x) \le w_* (y) + w_* (x-y) - s \, \min (w_* (y), w_* (x-y))
\end{equation}
holds for all $x,y  \ge 0$.
\end{lem}

\begin{proof}
{\em Step 1.} Preliminaries. We need some auxiliary functions and their basic properties. Obviously we have
\[ w'_* (t) = \frac{t}{\langle t \rangle_*^2} ( 1+ \log\log \langle t\rangle_*)\, , \qquad t >0\, ,  \]
and
\[ w_*'' (t) = \frac{1}{\langle t\rangle_*^2} (1 + \log\log \langle t\rangle_*) + 
\frac{t^2}{\langle t\rangle_*^4} \Big( \frac{1}{\log \langle t\rangle_*} - 2 - 2\log\log \langle t\rangle_* \Big)\, , \qquad t >0. 
\]
{~}\\
\begin{minipage}[u]{7cm}
\includegraphics[height=6cm]{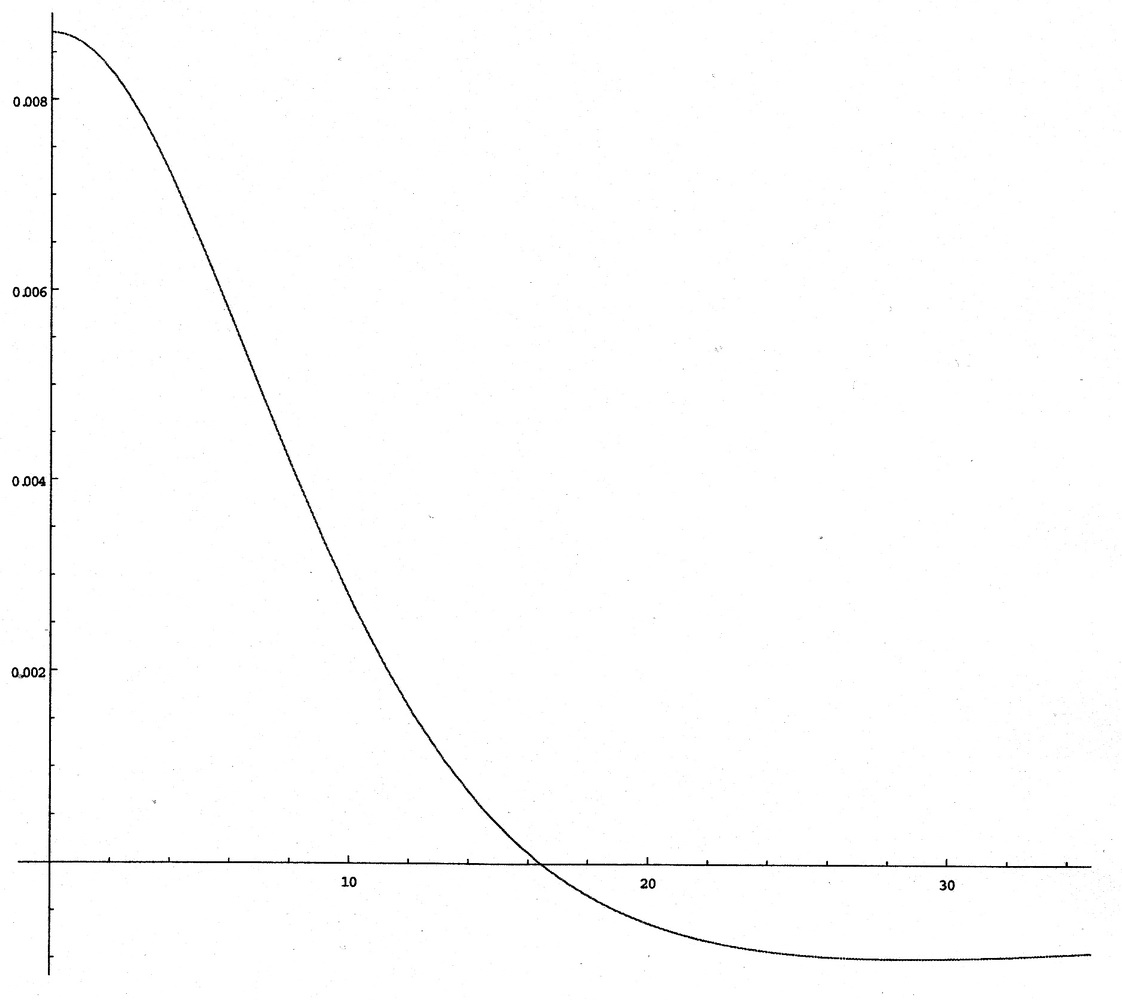}
\\
\centerline{Figure 1}
\end{minipage}
 \hfill
\begin{minipage}[u]{4.5cm}
The plot of $w_*'' $, see Figure 1, makes clear that  there exists a positive number $t_0$ such that
\begin{eqnarray*}
	w_*'' (t_0) & = & 0, \\
	w_*'' (t) & < & 0 \quad \mbox{ for } t>t_0, \\
	w_*'' (t) & > & 0 \quad \mbox{ for } 0 < t < t_0.
\end{eqnarray*}
A closer look shows that $t_0 \in (16.4449, 16.4451) $.
Thus, $w'_*$ has it's global maximum at $t=t_0$. 
\end{minipage}
\\
{~}\\
{~}\\
\noindent
Later on we shall need the  following function
\begin{eqnarray*}
p(t) & := &  \frac{t}{w_*(t)} w'_*(t) 
\\
& = &
\frac{t}{\log \langle t\rangle_* \log\log \langle t\rangle_*}  \frac{t}{\langle t \rangle_*^2} ( 1+ \log\log \langle t\rangle_*) \\
	& = & \frac{t^2}{\langle t \rangle_*^2} \frac{1+ \log\log \langle t\rangle_*}{\log  \langle t\rangle_* \log\log  \langle t\rangle_*}.
\end{eqnarray*}
Clearly, $p_0 := \sup_{t>0} p(t) < 1$. Figure 2 shows that 
$p_0 \in (0,410247, 0,410248)$.
In addition we need  
\[ q(t) := \frac{t}{w_*(t)}\, , \qquad t>0\, . \]
\begin{minipage}[u]{6cm}
\includegraphics[height=5cm]{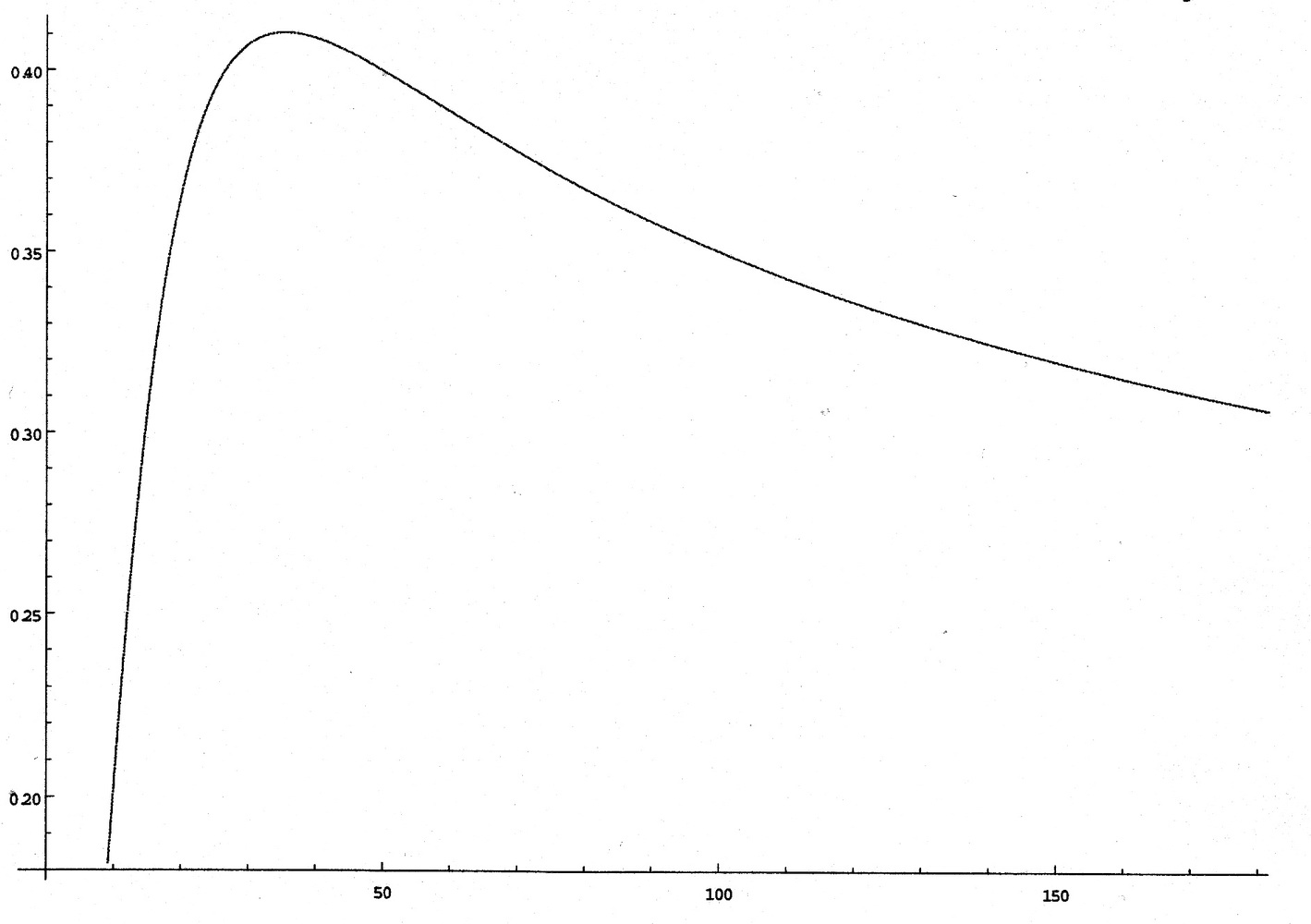}
\\
\centerline{Figure 2}
\end{minipage}
 \hfill
\begin{minipage}[u]{7cm}
\includegraphics[height=5cm]{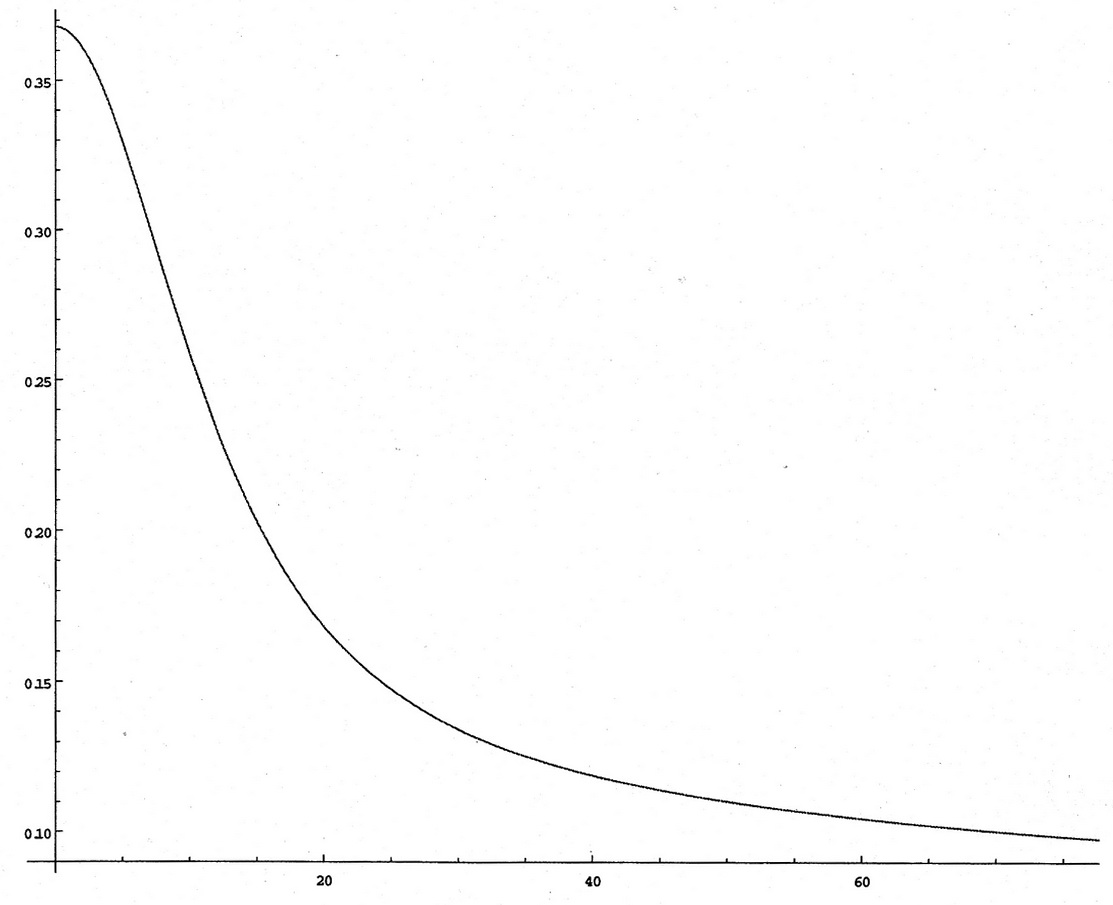}
\\
\centerline{Figure 3}
\end{minipage}
\\
{~}\\
{~}\\
\noindent
Considering  the derivative $q'$, see Figure 3,  it becomes clear that $q$ is strictly increasing for all $t\geq 0$. 
\\	
{\em Step 2.} We shall prove \eqref{ws-20} in case  $y \ge x$. From $w_*$ increasing  and $\min (w_* (y), \\ w_* (x-y)) = w_* (x-y)$
 we derive the validity of \eqref{ws-20} for all $s$,  $0 < s\le 1$. 
\\
{\em Step 3.} Now we turn to the case $x>y$. We shall split our investigations into the three 
cases $x\geq 2 t_0$, $t_0 < x <2 t_0$ and $0\leq x \leq t_0$, where $t_0$ is the root of $w_*''$. \\
 
\begin{minipage}[u]{10cm}
\includegraphics[height=8cm]{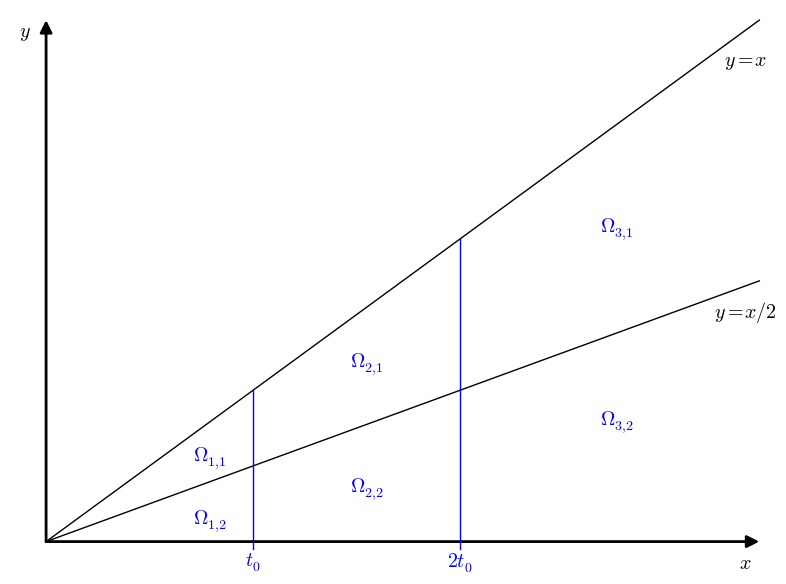}
\\
\centerline{Figure 4}
\end{minipage}
\\
{~}\\
{~}\\
\noindent
Figure 4 shows the division of the $(x,y)$-plane into six parts. 
In the subsequent computations we will refer to the zones $\Omega_{i,j}$, $i,j=1,2,3$.

\noindent
{\em Step 3.1.} Let $x \geq 2 t_0$. Now we have to consider two cases, i.e.,  $0 \le  y \le  x \le 2y$ ($(x,y)\in \Omega_{3,1}$) and 
$0\le 2y <x$ ($(x,y)\in \Omega_{3,2}$). \\
{\em Substep 3.1.1.} Suppose $(x,y)\in \Omega_{3,1}$, i.e., $\min (w_* (y), w_* (x-y)) = w_* (x-y)$. We consider the function
\[ 
h(x,y) := w_* (x) - w_*(y) - (1-s) w_*(x-y) \, , \qquad 0< y<x\, .
\]
Looking for extreme values of $h$ we have to determine the solutions of $\grad h =0$.
We obtain
 \begin{eqnarray*}
 	w'_*(x) - (1-s) w'_* (x-y) & = & 0, \\
	-w'_*(y) + (1-s) w'_*(x-y) & = & 0,
 \end{eqnarray*}
 which immediately  yields $w'_*(x) = w'_*(y)$.  Since  $w'(t)$ is strictly monotonically decreasing for all $t\geq t_0$ we conclude $x=y$. 
However,  this results in  $w'_*(y)=w'_*(x)=0$ since $w'_*(0)=0$. Because of $w'_*(t)>0$ for all $t>0$
we get a contradiction to $x \ge 2 t_0$. Thus, the maximum of $h=h(x,y)$ with respect to $\Omega_{3,1}$ lies on the boundary of the 
domain, that is, $x=2 t_0$, $y=x$ and $y=\frac{x}{2}$. Hence we need to check whether the following functions are nonpositive:
 \begin{eqnarray}
  h(2 t_0, y)  & = &  w_*(2 t_0) - w_*(y) - (1-s) w_*(2 t_0 -y) \quad \mbox{if}\quad  t_0 \leq y \leq 2 t_0, \qquad\label{c1-1-1} \\
	h(x,x)  & = & - (1-s) w_*(0) \quad \mbox{if}\quad 2 t_0 \leq x,  \label{c1-1-2} \\
	h(x,\frac{x}{2}) & = & w_*(x) - (2-s) w_*(\frac{x}{2}) \quad \mbox{if} \quad 2t_0 \leq x. \label{c1-1-3}
\end{eqnarray}
The function \eqref{c1-1-2} is trivially nonpositive for all $s\in (0,1]$. Considering the function \eqref{c1-1-1} 
and taking account of the mean value theorem we obtain
\[
 h(2 t_0, y) \le 0 \qquad \Longleftrightarrow \qquad 
w'_*(\xi) \leq (1-s) \frac{w_*(2t_0-y)}{2t_0 -y}, 
\]
where $\xi\in (y,2t_0)$. A substitution yields
\begin{equation} \label{boundEq}
	w'_*(\xi) \leq (1-s) \frac{w_*(z)}{z}
\end{equation}
with $0 \leq z \leq t_0$ and $\xi\in (2t_0-z, 2t_0)$. Using the monotonicity of $q$ and of $w'_*$, see Step 1, we find
\begin{equation} \label{w-01}
q(z)\, w'_*(\xi) = \frac{z}{w_*(z)} w'_*(\xi) \leq \frac{t_0}{w_*(t_0)} w'_*(\xi) \leq \frac{t_0}{w_*(t_0)} w'_*(t_0) \leq p_0 <1 \, .
\end{equation}
Thus, \eqref{boundEq} is true with  $0 < s \le  1-p_0$ and therefore $h(2 t_0, y)\le 0$ on  $[t_0,2t_0]$. 
\\
It is left to show that the function \eqref{c1-1-3} is nonpositive. First of all observe that
\[ 
  w'_*(\frac{x}{2}) \leq (1-s) \frac{w_*(\frac{x}{2})}{\frac{x}{2}}, \qquad x \ge 2t_0\, , 
\]
would imply $h(x, \frac x2) \le 0$ (we used the mean value theorem and the fact that $w'_*(t)$ is monotonically decreasing for all $t\geq t_0$). 
Equivalently we can write
\[ p(\frac t2)= \frac{\frac{x}{2} w'_*(\frac{x}{2})}{w_*(\frac{x}{2})} \leq 1-s. \]
As above this is true for all $0 < s \le  1-p_0$. \\
 {\em Substep 3.1.2.} Suppose $(x,y)\in \Omega_{3,2}$, i.e., $\min (w_* (y), w_* (x-y)) = w_* (y)$. Therefore we consider the function
  \[ h(x,y) := w_* (x) - w_*(x-y) - (1-s) w_*(y)\, , \qquad x \ge 2t_0\, , \: 0 < y \le \frac x2\, .  \]
 Considering the equation $\grad h =0 $ we obtain
 \begin{eqnarray*}
 	w'_*(x) - w'_*(x-y) & = & 0, \\
	w'_*(x-y) - (1-s) w'_*(y) & = & 0.
 \end{eqnarray*}
Taking care of the first equation only we get $w'_*(x) = w'_*(x-y)$. 
 Due to the assumptions we also know that $x \ge x-y \geq t_0$. From the monotonicity of $w'$ on $[t_0, \infty)$  we conclude 
that $y=0$ and then, by means of the second equation, $w'_*(x) = 0$. 
Again it follows $x=0$ which is in contradiction with the assumption $x\geq 2 t_0$. 
Thus, $h=h(x,y)$ attains it's maximum on the boundary of $\Omega_{3,2}$. We have to consider
 \begin{eqnarray}
 	h(2 t_0, y) & = & w_*(2 t_0) - w_*(2 t_0 - y) - (1-s) w_*(y)\quad  \mbox{if}\quad  0 \leq y < t_0, \qquad\label{c1-2-1} \\
	h(x,0) & = & -(1-s) w_*(0) \quad \mbox{if} \quad 2t_0 \leq x, \label{c1-2-2} \\
	h(x,\frac{x}{2}) & = & w_*(x) - (2-s) w_*(\frac{x}{2}) \quad \mbox{if} \quad 2t_0 \leq x. \label{c1-2-3}
\end{eqnarray}
The functions in \eqref{c1-2-2} and \eqref{c1-2-3} were already considered in Substep 3.1.1. 
It remains to show that the function \eqref{c1-2-1} is nonpositive. Applying the same arguments as in Substep 3.1.1 we get 
\[ h(2t_0,y) = y w'_*(\xi) - (1-s) w_*(y) \leq 0 \, , \qquad 2t_0 - y < \xi < 2t_0 \, .\]
This would follow from  
\[ \frac{y}{w_*(y)} w'_*(\xi) \leq 1-s, \]
if the latter inequality would be true for all  $y\in [0, t_0)$ and all $\xi\in (2t_0-y,2t_0)$. Arguing as in \eqref{w-01}
we find 
\[ \frac{y}{w_*(y)} w'_*(\xi) \leq \frac{t_0}{w_*(t_0)} w'_*(\xi) \leq \frac{t_0}{w_*(t_0)} w'_*(t_0) \leq p_0\, , \]
which shows that $h(2 t_0, y)$ is nonpositive if  $0 < s \le  1-p_0$. 
\\
{\em Step 3.2.} Let $t_0 < x < 2 t_0$. We cannot apply the strategy of Step 1 anymore 
since $w_*''=w_*''(y)$ changes sign in the considered interval for $y$.  
\\
{\em Substep 3.2.1.} Suppose $(x,y) \in \Omega_{2,1}$, i.e., $\min (w_* (y), w_* (x-y)) = w_* (x-y)$. Therefore we consider the function
 \[ h(x,y) := w_* (x) - w_*(y) - (1-s) w_*(x-y), \qquad t_0 < x < 2 t_0\, , \: \frac x2 < y < x \, ,\]
which is nonpositive if and only if
\begin{equation} \label{auxEq1}
	 \frac{w_*(x) - w_*(y)}{w_*(x-y)} \stackrel{[z=x-y]}{=} \frac{w_*(z+y)-w_*(y)}{w_*(z)} \leq 1-s. 
\end{equation}
By the mean value theorem we obtain
\[ \frac{w_*(z+y)-w_*(y)}{w_*(z)}  = \frac{z}{w_* (z)} w'_*(\xi_{y,z}), \]
where $z\in [0,t_0 )$ and $\xi_{y,z} \in (y, z+y)$. We know that $w_*'' = w_*''(t)$ vanishes if and only if $t = t_0$. 
Thus, $w'_* = w'_*(t)$ attains it's global maximum at $t=t_0$. The monotonicity of the function $q$ leads to 
\[ \frac{z}{w_* (z)} w'_*(\xi_{y,z}) \leq \frac{z}{w_* (z)} w'_*(t_0) \leq \frac{t_0}{w_* (t_0)} w'_*(t_0) \leq p_0\, . \]
Hence $h(x,y)\le 0$ if  $0 < s \le 1-p_0$. \\
{\em Substep 3.2.2.} Suppose $(x,y) \in \Omega_{2,2}$, i.e., $\min (w_* (y), w_* (x-y)) = w_* (y)$. Therefore we consider the function
 \[ h(x,y) := w_* (x) - w_*(x-y) - (1-s) w_*(y),\qquad t_0 < x < 2t_0\, , \: 0 < y < \frac x2\, ,  \]
which is nonpositive if and only if
\begin{equation} \label{auxEq2}
	 \frac{w_*(x) - w_*(x-y)}{w_*(y)} = \frac{y}{w_*(y)} w'_*(\xi_{x,y}) \leq 1-s
\end{equation}
with $\xi_{x,y} \in (x-y,x)$. By the same arguments as in Substep 3.2.1 we get $h(x,y)\le 0$ if  $0 < s \le 1-p_0$. 
\\
 {\em Step 3.3.} Let $0 \leq x \leq t_0$. Remark that all the substeps are treated as in Step 1. with the difference in the interval with respect to the $x$-variable. \\
 {\em Substep 3.3.1.} Suppose $(x,y) \in \Omega_{1,1}$, i.e., $\min (w_* (y), w_* (x-y)) = w_* (x-y)$. Therefore we consider the function
 \[ 
h(x,y) := w_* (x) - w_*(y) - (1-s) w_*(x-y)\, , \qquad 0 < x < t_0, \: \frac x2 < y < x. 
\]
Analogously to Step 3.1.1, equation $\grad h =0$ yields $x=y=0$. Consequently the maximal value of $h$ has to be located at the boundary.
It remains to consider the functions $h(x,x)$ and $h(x,\frac x2)$. The first one is trivial.
Concerning the second one we observe
\[ h(x,y) \leq  0 \qquad \Longleftrightarrow \qquad \frac x2 \, \frac{w'_* (\xi)}{w_* (\frac x2)} \le 1-s \, .\]
Employing the monotonicity properties of $p$ and $q$ we get
\[
\frac x2 \, \frac{w'_* (\xi)}{w_* (\frac x2)} \le t_0 \, \frac{w'_* (\xi)}{w_* (t_0)} \le p(t_0) \le p_0\, .
\]
Hence,  $h(x,y)\le 0$ if  $0 < s \le 1-p_0$. 
\\
{\em Step 3.3.2.} Suppose $(x,y) \in \Omega_{1,2}$, i.e., $\min (w_* (y), w_* (x-y)) = w_* (y)$. Therefore we consider the function
 \[ h(x,y) := w_* (x) - w_*(x-y) - (1-s) w_*(y)\, , \qquad 0< x < t_0, \: 0 < y < \frac{t_0}{2} . \]
Again  the equation $\grad h =0$ yields $x=y=0$. 
This justifies to study $h$ on $\partial \Omega_{1,2}$ only. 
The behavior of $h(x,\frac x2)$ we have already investigated in Substep 3.3.1. 
Since the function $h(x,0) \leq 0$ for all $s\in (0,1]$ we are done. 
\end{proof}

 \begin{rem}
Our proof yields that any $s$, $0< s \le 1- p_0$ will do the job. Here
$p_0 = \sup_{t>0} p(t) \in (0,410247, 0,410248)$.
 \end{rem}

Based on Lemma \ref{numer} we may follow the arguments in the proof of 
Theorem \ref{algebra} and Corollary \ref{algebra2}.

\begin{thm} \label{algebraUltra}
Let $1\leq p_1,p_2,q \leq \infty$. Define $p$ by $\frac 1p := \frac{1}{p_1}+\frac{1}{p_2}$.
Assume that $f\in\mathcal{U M}^{p_1,q}(\R^n)$ and $g\in\mathcal{U M}^{p_2,q}(\R^n)$. Then $f\cdot g\in\mathcal{U M}^{p,q}(\R^n)$ 
and it holds
\[ \|f \cdot g\|_{\mathcal{U M}^{p,q}} \leq C \|f\|_{\mathcal{U M}^{p_1,q}} \|g\|_{\mathcal{U M}^{p_2,q}} \]
with a positive constant $C$ which only depends on the choice of the frequency-uniform decomposition, 
the dimension $n$ and the parameter $q$. 
\\
In particular, the modulation space $\mathcal{U M}^{p,q}(\R^n)$ is an algebra under pointwise multiplication.
\end{thm}

\begin{proof}
{\em Step 1.}
We can basically follow the proof of Theorem \ref{algebra}. Therefore we will restrict ourselves to very few comments.
The first comment concerns the constant $c_3$.
Let $s$ denote the positive constant in Lemma \ref{numer}.
This time we define
\begin{equation}\label{ws-26}
c_3 := \Big( \sum_{m \in\Z^n}  e^{-s w_* (|m|) q'} \Big)^{\frac{1}{q'}} <\infty,
\end{equation}
see \eqref{ws-25}. Also the constant $c_4$, see \eqref{ws-27}, requires some changes.
The uniform boundedness and the positivity of $w_*'$ and the mean-value theorem imply
the existence of some $\xi$ such that
\[
e^{w_* (|j|) - w_* (|t-j|)} = e^{w_*' (\xi)(|j|-|j-t|)}\le e^{\|\, w_*' \|_{L^\infty}\, |t|} \, .
\]
Hence,
\begin{equation}\label{ws-28} 
\max_{\substack{t \in \Z^n, \\ -3<t_i < 3, \\ i=1,\ldots, n}} \sup_{j\in\Z^n} 
e^{w_* (|j|) - w_* (|t-j|)} \leq e^{\|w'_*\|_{L^\infty} \, 2\sqrt{n}} =: c_4 < \infty \, . 
\end{equation}
All the remaining arguments do not change.
\\
{\em Step 2.} We prove the algebra property of $\mathcal{U M}^{p,q}(\R^n)$.
An application of Lemma \ref{nikolskij} yields the continuous embedding
\[
\mathcal{U M}^{p_1,q}(\R^n)  \hookrightarrow \mathcal{U M}^{p_2,q}(\R^n)\, , \qquad 1\le p_1 \le p_2\le \infty\, .
\]
Now we apply Step 1 with $p_1=p_2 =2p$ and use the embedding $\mathcal{U M}^{p,q}(\R^n)  \hookrightarrow \mathcal{U M}^{p_j,q}(\R^n)$, $j=1,2$.
The proof is complete.
\end{proof}

Now the same steps as in Section \ref{Gevreysuperposition} yield a non-analytic superposition result on the modulation 
spaces $\mathcal{U M}^{p,q}(\R^n)$. First of all we need to show the subalgebra property. 
Therefore recall the decomposition of the phase space in $(2^n+1)$ parts as in Section \ref{Gevreysuperposition}.

\begin{prop} \label{subalgebraUltra}
Let $1\leq p, q \leq \infty$ and $N \in \N$.  Suppose that $\epsilon=(\epsilon_1,\ldots, \epsilon_n)$ is fixed with 
$\epsilon_j \in \{0,1\}$, $j=1,\ldots,n$. Let $s$ be as in Lemma \ref{numer} and $R\geq 2$.
The spaces
\[ \mathcal{U M}^{p,q}(\epsilon, R) := \{ f\in \mathcal{U M}^{p,q}(\R^n): \supp \F(f) \subset P_R(\epsilon) \} \]
are subalgebras of $\mathcal{U M}^{p,q}$. Furthermore, it holds
\begin{equation}\label{ws-29}
\|f g\|_{\mathcal{U M}^{p,q}} \leq G_{R,N} \, \|f\|_{\mathcal{U M}^{p,q}} \|g\|_{\mathcal{U M}^{p,q}} 
\end{equation}
	for all $f,g \in\mathcal{U M}^{p,q}(\epsilon, R)$. The constant $G_{R,N}$ can be specified by
	\[ G_{R,N} = C_N \, R^{-N}, \]
	where the constant $C_N > 0 $ depends  on $n,p,q$ and  $N$, but not on $f,g$ and $R$.
\end{prop}

\begin{proof}
We proceed as in proof of Proposition  \ref{subalgebra1} (using also the notation introduced there). 
Hence, it will be enough to discuss the changes caused by the replacement of 
$|k|^{1/s}$ by $w_*(|k|)$.
To estimate the counterparts of $S_{1,t,k}$ and $S_{2,t,k}$ we need to estimate 
the quantity
\[
\widetilde{G}_R:= \Big(\sum_{\substack{l\in\Z^n: ~ l, t-(l-k)\in P_R^* (0,\,  \ldots \, ,0), \\
 |l|\leq |l-k|}} e^{-s q' w_* (|l|)} \Big)^{1/q'} \, . 
\]
Here we will use that for any natural number $M$ there exists a constant $c_M$
such that 
\[
\sup_{l\in\Z^n} e^{-s\,  w_* (|l|)} \, (1+|l|)^M =:c_M < \infty\, .
\]
By definition of  $P_R^* (0,\,  \ldots \, ,0)$ it follows
\begin{eqnarray*}
\widetilde{G}_R & \le & c_M \, \Big(\sum_{l\in\Z^n: ~ \|l\|_\infty>R-1}  (1+|l|)^{-Mq'}  \Big)^{1/q'}
\\
&\le &
c_M \, \Big(\int_{\|x\|_\infty >R-2}   (1+|x|)^{-Mq'}dx  \Big)^{1/q'}
\\
&\le &
c_M \, \Big(\int_{|x| >R-2}   (1+|x|)^{-Mq'} dx \Big)^{1/q'}\, .
\end{eqnarray*}
Clearly, if $Mq'>n$ we obtain 
\begin{eqnarray*}
\int_{|x| >R-2}   (1+|x|)^{-Mq'}\, dx & = &  
2\, \frac{\pi^{n/2}}{\Gamma (n/2)}\,  \int_{R-2}^\infty r^{n-1} \, (1+r)^{-Mq'} \, dr
\\
 & \le  &  
2\, \frac{\pi^{n/2}}{\Gamma (n/2)}\,  \int_{R-1}^\infty  \, r^{n-1-Mq'} \, dr
\\
& = & 2\, \frac{\pi^{n/2}}{\Gamma (n/2)}\, \frac{(R-1)^{n-Mq'}}{Mq'-n}\, .
\end{eqnarray*}
Altogether this results in the estimate
\[
\widetilde{G}_R \le C \, (R-1)^{\frac{n}{q'}-M}\, ,
\]
where $C= C(n,q,M)$ depends on $n,q,M$ but not on $f,g$ and $R$.
For given $N \in \N$, by choosing $M -\frac{n}{q'}\ge N$, we can always guarantee an estimate of 
$\widetilde{G}_{R}$ by a constants times $R^{-N}$ for all $R\geq 2$. 
Taking into account the other constants showing up during the proof of 
 Proposition  \ref{subalgebra1} (but see also the proofs of Theorem \ref{algebra} and Corollary \ref{algebra2})
we finally get \eqref{ws-29}.
\end{proof}

\begin{lem} \label{estSuperUltra}
Let $1<p <\infty$ and  $1\le q \le \infty$.  Then, for any $\vartheta >1$ and for any $N \in \N$
there exist  positive constants $b,c$ such that
\[ \| e^{iu}-1\|_{\mathcal{U M}^{p,q}} \leq c \, \|u\|_{\mathcal{U M}^{p,q}}\, 
\left\{
\begin{array}{lll}
 e^{\vartheta \, w_* (b\, \|u\|^{1+\frac 1N}_{\mathcal{U M}^{p,q}})} &\qquad & \mbox{ if } \|u\|_{\mathcal{U M}^{p,q}} > 1, \\
1 &&  \mbox{ if } \|u\|_{\mathcal{U M}^{p,q}} \leq 1
\end{array}\right. 
\]
holds for all $u \in \mathcal{U M}^{p,q}(\R^n)$.
In addition, the constant $b$ can be chosen independent of $\vartheta$ and $N$.
\end{lem}

\begin{proof}
We follow the proof of Lemma \ref{estSuper}. 
\\
{\em Step 1.}
Thus, we let $u\in \mathcal{U M}^{p,q}(\R^n)$ satisfying $\supp \F (u) \subset P_R$. Again we employ the Taylor expansion of 
$e^{\i u}$ to get the norm estimate
\[ 
\| e^{\i u}-1\|_{\mathcal{U M}^{p,q}} \leq \Big\| \sum_{l=1}^r \frac{(\i u)^l}{l!} 
\Big\|_{\mathcal{U M}^{p,q}} + \Big\| \sum_{l=r+1}^{\infty} \frac{(\i u)^l}{l!} \Big\|_{\mathcal{U M}^{p,q}} =: S_1 + S_2. 
\]
After some computations analogously to the proof of Lemma \ref{estSuper}, see in particular \eqref{ws-08}, we obtain
\[ S_2 \le C_2\,  \|u\|_{\mathcal{U M}^{p,q}}\, , \qquad C_2:= \Big(C_1 \, e, \frac{3}{3-e}\Big), \]
where this time $C_1$ means the algebra constant with respect to $\mathcal{U M}^{p,q}(\R^n)$, see Theorem \ref{algebraUltra}.
Concerning $S_1$ we conlude
\begin{eqnarray*}
S_1 & \leq & \Big( \sum_{\substack{k\in\Z^n, \\ -Rr-1<k_i< Rr+1, \\ i=1,\ldots,n}} e^{w_*(|k|)q} 
\|\Box_k (e^{\i u}-1) \|_{L^p}^q \Big)^{\frac{1}{q}} + S_2\, .
\end{eqnarray*}
Because of 
\begin{eqnarray*}
\Big( \sum_{\substack{k\in\Z^n, \\ -Rr-1<k_i< Rr+1, \\ i=1,\ldots,n}} e^{w_*(|k|)q} \Big)^{\frac{1}{q}} 
\le c_1 \, e^{w_*(\sqrt{n}(Rr+1))}\, (\sqrt{n}(Rr+1))^{n/q}
\end{eqnarray*}
(here $c_1$ depends on $n$ only)
we find
\[ 
S_1 \le (2 \, C_3\, C_4\,  C_5 \, e^{\vartheta\, w_*(\sqrt{n}(Rr+1))} + C_2) \, \|u\|_{\mathcal{U M}^{p,q}},
\]
where $C_2$ has the above meaning and  $C_3$ has the original meaning from the proof
of Lemma \ref{estSuper}. The definitions of other two constants have to be modified.  $C_5 $ is defined as
\[
C_5 := \Big( \sum_{k\in\Z^n} e^{-w_*(|k|)q'} \Big)^{\frac{1}{q'}}< \infty\, ,
\] 
whereas  $C_4$ is now given by 
\[
C_4 := \sup_{r \in \N} \, \sup_{R \geq 2}\, c_1 \,   n^{\frac{n}{2q}}\, (Rr+1)^{n/q}
 e^{- (\vartheta-1)\, w_*(\sqrt{n}(Rr+1))} < \infty \, .
\]
Recall that we had choosen $r$ such that 
\[ 
3\, C_1\, \|u\|_{\mathcal{U M}^{p,q}} \leq r \leq 3\, C_1\,  \|u\|_{\mathcal{U M}^{p,q}} +1 \, , 
\]
see \eqref{ws-11}.
By means of this we can rewrite our estimate of $\| e^{\i u}-1\|_{\mathcal{U M}^{p,q}} $
and get
\begin{equation} \label{ws-30}
\| e^{\i u}-1 \|_{\mathcal{U M}^{p,q}} \leq c_0 
\|u\|_{\mathcal{U M}^{p,q}} \left( 1 + e^{\vartheta\, w_* ( b_0\, R\, \| u\|_{\mathcal{U M}^{p,q}})}
 \right)\, , 
\end{equation}
compare with \eqref{expest}, 
valid for all $u\in \mathcal{U M}^{p,q}(\R^n)$ satisfying $\supp \F(u) \subset P_R$ and with positive constants $b_0,c_0$ depending on $n,p$ 
and $q$ but independent of $u, r$ and $R$. 
\\
{\em Step 2.} Now we consider the general case. Let  $u\in \mathcal{U M}^{p,q}(\R^n)$. Again we decompose $u$ 
in the phase space according to Lemma \ref{estSuper}. Let $G_{R,N}$ be the constant in \eqref{ws-29}  and let $C_6$ be
as in proof of Lemma \ref{estSuper}. Then  it follows
\begin{eqnarray}
\label{ws-31b}
\| e^{\i u_{j_k}} -1\|_{\mathcal{U M}^{p,q}} & = & 
\Big\|\sum_{l=1}^\infty \frac{(\i u_{j_k})^l}{l!} \, \Big\|_{\mathcal{U M}^{p,q}}
\le 	\frac{1}{G_{R,N}} \Big(e^{G_{R,N}\,  \|u_{j_k}\|_{\mathcal{U M}^{p,q}}} -1 \Big) 
\nonumber 
\\
& \leq & \frac{1}{G_{R,N}} \Big( e^{G_{R,N}\,C_6\,  \|u\|_{\mathcal{U M}^{p,q}}} -1 \Big)\, , 
\end{eqnarray}
see \eqref{est1}, as well as
\begin{eqnarray}\label{ws-32b}
\| e^{\i u_0} -1\|_{\mathcal{U M}^{p,q}} & \leq &  c_0 \, C_6\,  \|u\|_{\mathcal{U M}^{p,q}} 
\Big( 1 + e^{\vartheta \, w_* (b_0 \, R\,  \| u\|_{\mathcal{U M}^{p,q}}) } \Big)\, ,  
\end{eqnarray}
see \eqref{est2}.
\\
{\em Substep 2.1.} Let  $\|u\|_{\mathcal{UM}^{p,q}} \leq 1$. We choose $R=3$. 
As in proof of Lemma \ref{estSuper} we conclude
\begin{equation}\label{ws-33}
 \|e^{\i u} -1\|_{\mathcal{UM}^{p,q}} \leq c_2 \, \|u\|_{\mathcal{UM}^{p,q}}, 
\end{equation}
where $c_2$ does not depend on $u$.\\
{\em Substep 2.2.} Let $\|u\|_{\mathcal{UM}^{p,q}} >1$. 
We know that the algebra constant $G_{R,N}$ in \eqref{ws-31b} 
is a function of $R$, more exactly,
\[
 G_{R,N} = \widetilde{C}_N \, R^{-N}\, .
\]
Taking into account that
	\begin{itemize}
		\item $G_{R,N}$ is strictly decreasing and positive  and
		\item $\displaystyle \lim_{R\to \infty} G_{R,N} = 0$
	\end{itemize}
we may choose $R>2$ according to 
\[ 
\Big(\frac{2}{R}\Big)^N = \frac{G_{R,N}}{G_{2,N}}  = \|u\|_{\mathcal{UM}^{p,q}}^{-1} \, .
\]
This means in particular
\[
G_{R,N}\,  \|u\|_{\mathcal{UM}^{p,q}} = \widetilde{C}_N \, 2^{-N}\qquad 
\mbox{and}\qquad R =   2\, \|u\|_{\mathcal{UM}^{p,q}}^{1/N} \, .
\]
Now \eqref{ws-31b}  and \eqref{ws-32b} result in  
\begin{eqnarray*}
&& \|e^{\i u} -1\|_{\mathcal{UM}^{p,q}}
\\
& \leq &  c_{3}\, \max_{\alpha \in \{0,1\}, \beta \in \{0, \ldots \, 2^n\}}
\, \bigg(c_0 \, C_6\,  \|u\|_{\mathcal{UM}^{p,q}} 
\Big( 1 + e^{\vartheta\, w_*(b_0\, R\, \| u\|_{\mathcal{UM}^{p,q}})} \Big) \bigg)^\alpha \\
& & \qquad \qquad \qquad \times \, \Big( \frac{e^{G_{R,N}\,C_6\,  \|u\|_{\mathcal{UM}^{p,q}}} -1}{G_{R,N}} \Big)^\beta
\\
& \le &  c_{4}\, \max_{\alpha \in \{0,1\}, \beta \in \{0, \ldots \, 2^n\}}
\, \bigg(\|u\|_{\mathcal{UM}^{p,q}} 
\Big( 1 + e^{\vartheta\,  w_* (b_0 \, 2\,  \| u\|_{\mathcal{UM}^{p,q}}^{1+\frac 1N})} \Big) \bigg)^\alpha \|u\|_{\mathcal{UM}^{p,q}}^\beta \\
& & \qquad \qquad \qquad \times \, \Big( \frac{e^{G_{2,N}\,C_6}\, -1}{G_{2,N}} \Big)^\beta
\\
&\le & c_5\, \|u\|_{\mathcal{UM}^{p,q}} \, \Big( 1 + 
e^{\vartheta \, w_* (b_0 \,2\,   \| u\|^{1+\frac 1N}_{\mathcal{UM}^{p,q}})} \Big)   
\end{eqnarray*}
with a constant $c_{5}$ depending on $\vartheta$ and $N$ but independent of $u$.
\end{proof}

\begin{lem} \label{ContExpUltra}
Let $1<p <\infty$ and  $1\le q \le \infty$.
Assume $u\in \mathcal{U M}^{p,q}(\R^n)$ to be fixed and define a function $g: \R \mapsto \mathcal{U M}^{p,q}(\R^n)$ by 
$g(\xi) = e^{\i u(x) \xi}-1$. Then the function $g$ is continuous.
\end{lem}

\begin{proof}
This lemma can be proved  in the same way as Lemma \ref{ContExp}.
\end{proof}

\begin{thm} \label{SuperpositionUltra}
Let $1<p <\infty$, $1\le q \le \infty$,  $\varepsilon >0$ and $\vartheta >1$.
		Let $\mu$ be a complex measure on $\R$ such that
\begin{equation} \label{FourierEstUltra}
L_1(\lambda) := \int_{\R} e^{\vartheta\,  \log (\langle \lambda \, |\xi|^{1+\varepsilon} \rangle_*)\,   
\log\log (\langle \lambda |\xi|^{1+\varepsilon} \rangle_*)} \, d|\mu| (\xi) < \infty
\end{equation}
for any $\lambda >0 $ and such that $\mu(\R) = 0$. 
\\
Furthermore, assume that the function $f$ is the inverse Fourier transform of $\mu$. 
Then $f\in C^\infty$ and the composition operator $T_f: u \mapsto f \circ u$ maps $\mathcal{U M}^{p,q}$ into $\mathcal{U M}^{p,q}$.
\end{thm}
	\begin{proof}
		We exactly follow the proof of Theorem \ref{Superposition}. Using equation \eqref{FourierEstUltra}, 
Lemma \ref{ContExpUltra} and Lemma \ref{estSuperUltra} complete the proof.
	\end{proof}

Note the following conclusion.
	\begin{cor} \label{SuperpositionCorUltra}
	Let $1<p <\infty$ and  $1\le q \le \infty$.
		Let $\mu$ be a complex measure on $\R$ with the corresponding bounded density 
function $g$, i.e., $d\mu (\xi) = g(\xi) \, d\xi$. Suppose that
		\begin{equation*} %\label{densCondUltra}
			\lim_{|\xi|\to \infty} \frac{\log (\langle |\xi| \rangle_*) 
\log\log (\langle |\xi| \rangle_*) }{\log |g(\xi)|} = 0
		\end{equation*}
		and $\displaystyle \int_{\R} d\mu (\xi) = \int_{\R} g(\xi) \, d\xi = 0$. 
Assume the function $f$ to be the inverse Fourier transform of $g$. Then $f\in C^\infty$ and the composition operator 
$T_f: u \mapsto f \circ u$ maps $\mathcal{U M}^{p,q}$ into $\mathcal{U M}^{p,q}$.
\end{cor}
\begin{proof}
Without loss of generality let $0<\varepsilon < 1$.
Making use of the elementary inequality 
\[
w_* (|\xi|^{1+\varepsilon}) \le (1+\varepsilon)\, \log (\langle |\xi| \rangle_*)\, \Big(\varepsilon + 
\log\log (\langle |\xi| \rangle_*)\Big) 
\]
we can follow the proof of  Corollary \ref{SuperpositionCor}. 
\end{proof}

\subsection*{One  Example} 

By $\mathcal{X}$ we denote the characteristic function of the interval $[0,1)$. Then we define the special function

\[
{\rm up} (x) := \lim_{r \to \infty} \Big(\mathcal{X} * (2\, \mathcal{X} (2\, \cdot \, ) * \ldots * (2^r\, \mathcal{X} (2^r\, \cdot \, )\Big)(x)
\, , \quad x \in \R\, ,
\]
originally introduced by the brothers Rvatchev \cite{RR}. A good source represents the survey article \cite{rv}. 
This function has a number of nice properties, e.g., 
\[
{\rm up}' (x) = 2 {\rm up} (2x+1) - 2{\rm up} (2x-1)\, , \quad x \in \R\, .
\]
Here we need that  ${\rm up} \in C_0^\infty (\R)$, ${\rm up} \ge 0$ on $\R$, 
$ {\rm up} (x)>0 $ if $ 0< x < 2$, and 
\[
\F {\rm up} (\xi) = \frac{e^{-i\xi}}{\sqrt{2\pi}} \, \prod_{j=1}^\infty \frac{\sin 2^{-j}\xi}{2^{-j}\xi}\, , \quad \xi \in \R\, .
\]
It is an exercise to prove the decay estimate
\begin{eqnarray*}
|\F {\rm up} (\xi)| & \le &  \frac{1}{\sqrt{2\pi}}\, |\xi|^{1- (\log_2 |\xi|)/2}
\\
 & =  &  \frac{1}{\sqrt{2\pi}}\, e^{(1- (\log_2 |\xi|)/2) \, \log |\xi|} \qquad \mbox{if}\quad |\xi|\ge 1\, .
\end{eqnarray*}

\begin{cor} \label{ex2}
Let $1<p <\infty$ and  $1\le q \le \infty$.
Then the  composition operator 
$T_{{\rm up}}: u \mapsto {\rm up} \circ u$ maps $\mathcal{U M}^{p,q}$ into $\mathcal{U M}^{p,q}$.
\end{cor}

%&&&&&&&&&&&&&&&&&&&&&&&&&&&&&&&&&&&&&&&&&&&&&&&&&&&&&&&&&&&&&&&&&&&&&&&&&&
%&&&&&&&&&&&&&&&&&&&&&&&&&&&&&&&&&&&&&&&&&&&&&&&&&&&&&&&&&&&&&&&&&&&&&&&&&

\section{Concluding Remarks and Open Questions} \label{finalcomments}

%&&&&&&&&&&&&&&&&&&&&&&&&&&&&&&&&&&&&&&&&&&&&&&&&&&&&&&&&&&&&&&&&&&&&&&&&&&
%&&&&&&&&&&&&&&&&&&&&&&&&&&&&&&&&&&&&&&&&&&&&&&&&&&&&&&&&&&&&&&&&&&&&&&&&&

\begin{itemize}
 \item 
The main issue of this paper has been  to explain non-analytic superposition on modulation spaces with subexponential  weights.
Up to now we have got results in two special situations, namely in Section \ref{Gevreysuperposition} we studied Gevrey-modulation spaces, in 
Section \ref{Ultradifferentiablesuperposition} we studied $\mathcal{UM}^{p,q}$. In a forthcoming project we would like to extend the theory 
of non-analytic superpositions to more general modulation spaces of ultra-differentiable type.  
\item
In an other  forthcoming project we will study superposition operators on modulation spaces of Sobolev type, too. 
Here two properties seem to be essential for this issue.

\begin{prop} \label{embedding}(\cite{Sugi})
		Let $1\leq p,q \leq \infty$. If $s >\frac{n}{q'}$, where $q'$ is the conjugate to $q$, then 
$M_{p,q}^{s} (\R^n)$ is an algebra with respect to pointwise multiplication. 
\end{prop}

\noindent The condition $s >\frac{n}{q'}$  also implies $M^s_{p,q} (\R^n) \subset L^\infty(\R^n)$, see \cite{Sugi}. 
In some sense we believe that for the study of superposition operators on modulation spaces of Sobolev type this condition is a threshold
for inner functions in a similar way as for the classical Sobolev spaces itself, see Bourdaud \cite{Bou2},  \cite[5.2.4]{runstSickel}
or Bourdaud, S. \cite{BS}.

\item
We are able to apply the superposition results to handle nonlinear partial differential equations.
By the concepts of \cite{brs} we can investigate the solutions of \eqref{Gevrey2} and \eqref{Gevrey4} which were introduced 
in Section \ref{introduction}. We expect local (in time) existence results for data belonging to suitable Gevrey-modulation spaces.
We can also apply the modulation spaces of Section \ref{Ultradifferentiablesuperposition} if the coefficient $a=a(t)$ in 
\eqref{Gevrey4} has a suitable modulus of continuity behavior. However within the scope of this paper only the relevant 
tools are provided. In future work we will study the existence of locally (in time) and globally (in time) solutions.
\end{itemize}

\subsection*{Acknowledgment}
The first author wants to express his gratitude to Professor Toft from Linnaeus University V\"axj\"o in Sweden 
for being his supervisor during the master studies and introducing him to the field of modulation spaces. 
Moreover, the first author gratefully acknowledge the fruitful and constructive discussions on this field 
with Professor Toft. 

%&&&&&&&&&&&&&&&&&&&&&&&&&&&&&&&&&&&&&&&&&&&&&&&&&&&&&&&&&&&&&&&&&&&&&&&&&&
%&&&&&&&&&&&&&&&&&&&&&&&&&&&&&&&&&&&&&&&&&&&&&&&&&&&&&&&&&&&&&&&&&&&&&&&&&

% ------------------------------------------------------------------------
\end{document}